\documentclass[11pt,oneside]{amsart}

\usepackage{amsmath,amsthm,amssymb,amsfonts,amscd}
\usepackage{url}
\usepackage{color}
\usepackage{setspace}
\usepackage{enumerate}

\usepackage{epsfig}
\usepackage{graphicx}
\usepackage{wrapfig}
\usepackage{mathrsfs}
\usepackage{pinlabel}

\definecolor{grey}{rgb}{0.7,0.7,0.7}

\newcounter{notes}%

\newcommand{\ignore}[1]{}

\theoremstyle{plain}
\newtheorem{theorem}{Theorem}
\newtheorem{proposition}[theorem]{Proposition}
\newtheorem{corollary}[theorem]{Corollary}
\newtheorem{observation}[theorem]{Observation}

\newtheorem{claim}[theorem]{Claim}

\newtheoremstyle{theoremwithref}{}{}{\itshape}{}{\bfseries}{.}{.5em}{#1 #2 #3}
\theoremstyle{theoremwithref}

\theoremstyle{definition}

\newtheorem{remark}[theorem]{Remark}
\newtheorem{remarks}[theorem]{Remarks}
\newtheorem{question}[theorem]{Question}

\numberwithin{theorem}{section}
\numberwithin{equation}{section}

\newcommand{\D}{\mathrm{d}}

\newcommand{\CP}{\mathcal{C}}

\newcommand{\ACP}{\mathsf{C}}
\newcommand{\ACHS}{\mathsf{H}}
\newcommand{\DD}{\mathsf{D}}
\newcommand{\SQ}{\mathsf{SQ}}

\newcommand{\NN}{\mathbb{N}}

\newcommand{\RR}{\mathbb{R}}

\newcommand{\HH}{\mathbb{H}}
\newcommand{\PP}{\mathbb{P}}

\newcommand{\OO}{\mathrm{O}}
\newcommand{\PO}{\mathrm{PO}}

\newcommand{\PSL}{\mathrm{PSL}}
\newcommand{\PGL}{\mathrm{PGL}}

\newcommand{\g}{\mathfrak{g}}

\newcommand{\AdS}{\mathrm{AdS}^3}
\newcommand{\AdSS}{\text{AdS}}

\newcommand{\A}{\mathscr{A}}

\newcommand{\Ad}{\operatorname{Ad}}

\newcommand{\Hom}{\mathrm{Hom}}

\newcommand{\ie}{i.e.\ }
\newcommand{\eg}{e.g.\ }
\newcommand{\resp}{resp.\ }

\title[Fundamental domains in anti-de Sitter $3$-space]{Fundamental domains for free groups acting on anti-de Sitter $3$-space}

\author{Jeffrey Danciger}
\address{Department of Mathematics, The University of Texas at Austin, 1 University Station C1200, Austin, TX 78712, USA}
\email{jdanciger@math.utexas.edu}

\author{Fran\c{c}ois Gu\'eritaud}
\address{CNRS and Universit\'e Lille 1, Laboratoire Paul Painlev\'e, 59655 Villeneuve d'Ascq Cedex, France 
\newline Wolfgang-Pauli Institute, University of Vienna, CNRS-UMI 2842, Austria}
\email{francois.gueritaud@math.univ-lille1.fr}

\author{Fanny Kassel}
\address{CNRS and Universit\'e Lille 1, Laboratoire Paul Painlev\'e, 59655 Villeneuve d'Ascq Cedex, France}
\email{fanny.kassel@math.univ-lille1.fr}

\thanks{J.D. was partially supported by the National Science Foundation under the grant DMS 1103939.
F.G. and F.K. were partially supported by the Agence Nationale de la Recherche under the grants DiscGroup (ANR-11-BS01-013) and ETTT (ANR-09-BLAN-0116-01), and through the Labex CEMPI (ANR-11-LABX-0007-01). 
All three authors acknowledge support from the GEAR Network, funded by the National Science Foundation under the grants DMS 1107452, 1107263, and 1107367 (``RNMS: GEometric structures And Representation varieties").}

\begin{document}

\begin{abstract}
Crooked planes are piecewise linear surfaces that were introduced by Drumm in the early 1990s to construct fundamental domains for properly discontinuous actions of free groups on Minkowski $3$-space. 
In a previous paper, we introduced analogues of these surfaces, called $\AdSS$ crooked planes, in the $3$-dimensio\-nal anti-de Sitter space $\AdS$; we showed that many properly discontinuous actions of free groups on $\AdS$ admit fundamental domains bounded by $\AdSS$ crooked planes.
Here we study further the question of which proper actions on $\AdS$ admit crooked fundamental domains, and show that some do not, in contrast to the Minkowski setting.
\end{abstract}

\maketitle
\tableofcontents

%%%%%%%%%%%%%%%%%%%%%%%%%%%%%%%%%%%%%%%%%%%%%%%%%%%
\section{Introduction}

This paper concerns the construction of fundamental domains for properly discontinuous group actions, by isometries, on a Lorentzian space, specifically the $3$-dimensional anti-de Sitter space $\AdS$.
In the setting of group actions on Riemannian spaces such as the Euclidean or the real hyperbolic space, there are standard procedures for constructing fundamental domains, \eg Dirichlet polyhedra.
However, such constructions make crucial use of the Riemannian metric, and typically cannot be made to work in the Lorentzian setting (nor in other nonmetric settings).
Hence the construction of fundamental domains in Lorentzian geometry is often more of an ad hoc endeavor.

In the setting of $3$-dimensional flat Lorentzian geometry, or Minkowski geometry, Drumm \cite{dru92} introduced piecewise linear surfaces called \emph{crooked planes} and used them to build polyhedral fundamental domains for \emph{Margulis spacetimes}.
These are the quotients of the Minkowski $3$-space $\RR^{2,1}$ (namely $\RR^3$ endowed with a translation-invariant metric of signature $(2,1)$) by free groups acting properly discontinuously by affine isometries.
The first such spacetimes were discovered by Margulis \cite{mar83, mar84} in 1983, as counterexamples to a possible extension \cite{mil77} of the Auslander Conjecture \cite{aus64}.
Since then, Margulis spacetimes and their fundamental domains have been the object of a rich literature, most prominently by Charette, Drumm, Goldman, Labourie, and Margulis: see \cite{dru92, dg95, dg99, cg00, fra03, glm09, cdg10,cdg11,cdg13, cg13, dgk13}.
Recently, we showed in \cite{dgk-strips, dgk-parab} that every Margulis spacetime $M$ comes (in an essentially unique way) from a so-called \emph{infinitesimal strip deformation} of a noncompact hyperbolic surfa\-ce, and that consequently $M$ admits a fundamental domain bounded by pairwise disjoint crooked planes.
This last result had been conjectured by Drumm and Goldman \cite{dg95}, and was known as the \emph{Crooked Plane Conjecture}.

The $3$-dimensional \emph{anti-de Sitter space} $\AdS$ is a model space for Lorentzian geometry of constant negative curvature. There is a close connection between the theory of complete $\AdSS$ spacetimes, specifically quotients of $\AdS$, and the theory of Margulis spacetimes.
In particular, we showed in \cite{dgk13, dgk-parab} that any Margulis spacetime is, in a certain sense, a rescaled limit of collapsing $\AdSS$ spacetimes.
It follows from more recent work~\cite{dgk-strips} (see Appendix~\ref{app:convergence}) that such families of collapsing $\AdSS$ spacetimes may be constructed explicitly by deforming a fundamental domain in Minkowski space, bounded by Drumm's crooked planes, into fundamental domains in $\AdS$.
Specifically, the collapsing spacetimes admit collapsing fundamental domains in $\AdS$ bounded by piecewise geodesic surfaces, which we call \emph{AdS crooked planes}, and the rescaled limit of these fundamental domains is the fundamental domain for the limiting Margulis spacetime. 
The construction actually implies that fundamental domains bounded by AdS crooked planes exist for a large open set in the moduli space of properly discontinuous actions of a free group on $\AdS$. The purpose of the present paper is to prove the following:

\begin{theorem}\label{thm:no-crooked-fd}
There exist properly discontinuous actions of finitely generated free groups on $\AdS$ that do not admit any fundamental domain in $\AdS$ bounded by pairwise disjoint $\AdSS$ crooked planes.
\end{theorem}
\noindent Thus, while the analogy between $3$-dimensional Minkowski and anti-de Sitter geometry is very strong, the analogue of the Crooked Plane Conjecture in $\AdS$ is false.
The question of which proper actions on $\AdS$ admit crooked fundamental domains (Question~\ref{q:fund-domain}) remains open.

Theorem~\ref{thm:no-crooked-fd} is proved in Section~\ref{sec:no-crooked-fd}.
Prior to this, we establish several useful results about $\AdSS$ crooked planes:
\begin{itemize}
  \item Theorems \ref{thm:AdS-disj-left} and~\ref{thm:AdS-disj-right} give complete disjointness criteria for $\AdSS$ crooked planes (see also Theorems \ref{thm:AdS-disj-SQ-left} and~\ref{thm:AdS-disj-SQ-right} in Appendix~\ref{app:SQ}).
  They should be compared with the disjointness criterion of Drumm--Goldman \cite{dg99, bcdg13} for crooked planes in~$\RR^{2,1}$.
  \item Theorem~\ref{thm:charact-crooked-domains} gives a necessary and sufficient condition, in terms of $2$-dimensional hyperbolic geometry, for the existence of a fundamental domain bounded by pairwise disjoint $\AdSS$ crooked planes.
  \item Corollary~\ref{cor:AdS-Drumm-shortening} is the analogue, in $\AdS$, of Drumm's result in~$\RR^{2,1}$ \cite{dru92, cg00} that gluing together a polyhedron $\mathsf{R}$ bounded by disjoint crooked planes yields a proper action (which one could call a \emph{Lorentzian Schottky group}) for which $\mathsf{R}$ is a fundamental domain.
\end{itemize}

%%%%%%%%%%%%%%%%%%%%%%%%%%%%%%%%%%%%%%%%%%%%%%%%%%%
\section{Reminders}

We first recall a few basic facts on $3$-dimensional anti-de Sitter geometry, as well as the definition of $\AdSS$ crooked planes introduced in \cite{dgk-strips}.

%%%%%%%%%%%%%%%%%%%%%%%%%%%%%%%%%%
\subsection{Anti-de Sitter $3$-space}\label{subsec:AdS3}

The $3$-dimensional anti-de Sitter space
$$\AdS = \PO(2,2)/\OO(2,1)$$
is a model space for Lorentzian $3$-manifolds of constant negative curvature.
It can be realized as the (open) set of negative points in $\PP^3(\RR)$ with respect to a quadratic form of signature $(2,2)$; its isometry group is $\PO(2,2)$.
The intersection of $\AdS$ with an affine chart of $\PP^3(\RR)$ is the region bounded by a one-sheeted hyperboloid.
Let $G=\PGL_2(\RR)$ be the group of isometries of the hyperbolic plane $\HH^2$, with identity component $G_0=\PSL_2(\RR)$ consisting of orientation-preserving isometries.
Then $\AdS$ can be realized as the group $G_0$ equipped with the biinvariant Lorentzian structure induced by half the Killing form of $\g=\mathfrak{pgl}_2(\RR)=\mathfrak{psl}_2(\RR)$; the group of orientation-preserving isometries of $\AdS$ then identifies with
$$(G\times G)_+ := \{ (g_1,g_2)\in G\times G~|~g_1g_2\in G_0\} ,$$
acting on~$G_0$ by right and left multiplication: $(g_1,g_2)\cdot g=g_2gg_1^{-1}$.
The boundary $\partial_{\infty}\AdS$ of $\AdS$ in $\PP^3(\RR)$ identifies with the projectivization of the set of $(2\times 2)$-matrices of rank~$1$, which itself identifies with $\PP^1(\RR)\times\PP^1(\RR)$ by taking the projectivized kernel and the projectivized image.
The action of $(G\times G)_+$ on $\AdS$ induces the natural action of $(G\times G)_+$ on
$$\PP^1(\RR)\times\PP^1(\RR)\simeq\partial_{\infty}\HH^2\times\partial_{\infty}\HH^2,$$
where $\partial_{\infty}\HH^2$ is the boundary at infinity of~$\HH^2$.

%%%%%%%%%%%%%%%%%%%%%%%%%%%%%%%%%%
\subsection{Quotients of $\AdS$}\label{subsec:prop-AdS}

By \cite{kr85}, any torsion-free discrete subgroup of $(G\times G)_+$ acting properly~dis\-continuously on $\AdS$ is necessarily, up to switching the two factors of $G\times G$, of the form
\begin{equation}\Gamma^{j,\rho} = \{ (j(\gamma), \rho(\gamma))~|~\gamma\in\Gamma\} \subset G\times G \label{eqn:def-jr}\end{equation}
where $\Gamma$ is a discrete group and $j,\rho\in\Hom(\Gamma,G)$ are two representations with $j$ injective and discrete.
Suppose that $\Gamma$ is finitely generated.
By \cite{kasPhD,gk13}, a necessary and sufficient condition for the action of $\Gamma^{j,\rho}$ on $\AdS$ to be properly discontinuous is that (up to switching the two factors) $j$ be injective and discrete and $\rho$ ``uniformly shorter'' than~$j$, in the sense that there exists a $(j,\rho)$-equivariant Lipschitz map $\HH^2\rightarrow\HH^2$ with Lipschitz constant $<1$.
When $j$ is convex cocompact or when the group $\rho(\Gamma)$ does not have a unique fixed point in $\partial_{\infty}\HH^2$, this property is also equivalent \cite{kasPhD}, \cite[Th.\,1.8]{gk13} to
\begin{equation}\label{eqn:propcritAdS}
\sup_{\gamma\in\Gamma,\ \lambda(j(\gamma))>0 }\ \frac{\lambda(\rho(\gamma))}{\lambda(j(\gamma))} < 1,
\end{equation}
where $\lambda : G\rightarrow\RR_+$ is the translation length function in~$\HH^2$:
\begin{equation}\label{eqn:def-lambda}
\lambda(g) := \inf_{p\in\HH^2} \, d(p,g\cdot p).
\end{equation}
(We denote by $d$ the hyperbolic metric on~$\HH^2$.)
Note that $\lambda(g)$ is the translation length of~$g$ if $g\in G$ is hyperbolic, and $0$ otherwise.

If $\Gamma$ is the fundamental group of a compact surface, and both $j$ and~$\rho$ are injective and discrete, then \eqref{eqn:propcritAdS} is never satisfied \cite{thu86}.
We are interested here in the case that $\Gamma$ is a (finitely generated) free group and we shall assume this throughout the remainder of the paper.
In this case, there are many pairs $(j,\rho)$ satisfying \eqref{eqn:propcritAdS} with both $j$ and $\rho$ injective and discrete. 
In \cite{dgk-strips}, we proved that if $(j, \rho)$ satisfies \eqref{eqn:propcritAdS} and $j$ and~$\rho$ are both convex cocompact, with $j(\Gamma)\backslash \HH^2$ and $\rho(\Gamma)\backslash \HH^2$ homeomorphic as marked surfaces, then $j$ is a so-called \emph{strip deformation} of~$\rho$ (in an essentially unique way): this means that $j(\Gamma)\backslash\HH^2$ is obtained from $\rho(\Gamma)\backslash\HH^2$ by adding in finitely many disjoint hyperbolic \emph{strips} (\ie regions isometric to a domain of $\HH^2$ bounded by two disjoint geodesic lines).
As an application, we showed \cite{dgk-strips} that in this case the group $\Gamma^{j,\rho}$ admits a fundamental domain bounded by pairwise disjoint \emph{$\AdSS$ crooked planes}.

%%%%%%%%%%%%%%%%%%%%%%%%%%%%%%%%%%%
\subsection{$\AdSS$ crooked planes}\label{subsec:ACP}

Let us recall the definition of $\AdSS$ crooked planes, introduced in \cite{dgk-strips}.
A crooked plane in $\AdS \simeq G_0$ is the union of three pieces:
\begin{itemize}
  \item a \emph{stem}, defined to be the union of all causal (\ie timelike or lightlike) geodesic lines of a given timelike plane of $\AdS$ that pass through a given point, called the \emph{center} of the $\AdSS$ crooked plane;
  \item two \emph{wings}, defined to be two disjoint open lightlike half-planes of $\AdS$ whose respective boundaries are the two (lightlike) boundary lines of the stem.
\end{itemize}

\begin{figure}[ht!]
\centering
\labellist
\small\hair 2pt
\pinlabel $\mathsf{S}(\ell)$ [r] at 93 25
\pinlabel $\mathsf{W}(v^+)$ [u] at 92 90
\pinlabel $\mathsf{W}(v^-)$ [d] at 53 64
\pinlabel $\mathsf{S}(\ell)$ [d] at 259 80
\pinlabel $\mathsf{W}(v^+)$ [d] at 293 122
\pinlabel $\mathsf{W}(v^-)$ [d] at 225 99
\endlabellist
\includegraphics[scale=0.8]{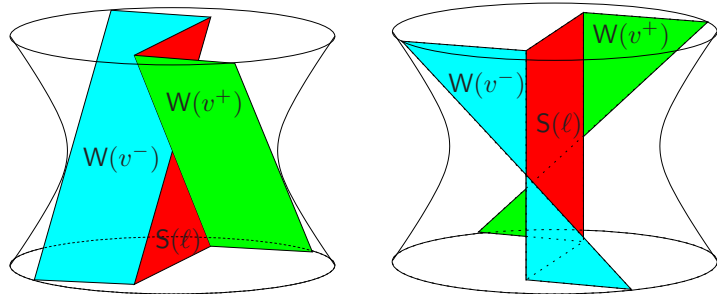}
\caption{The left $\AdSS$ crooked plane $\ACP(\ell)$, seen in two different affine charts of $\PP^3(\RR)\supset\AdS$. The stem $\mathsf{S}(\ell)$ (red) is a bigon whose closure in $\PP^3(\RR)$ meets the boundary of $\AdS$ in two points, tangentially. On the left, these two points are at infinity; on the right, the center of the stem is at infinity. Each wing $\mathsf{W}(v^+)$ or $\mathsf{W}(v^-)$ (green or blue) is itself a bigon, bounded by a line contained in the boundary of $\AdS$ and a lightlike line of the stem $\mathsf{S}(\ell)$.}
\label{fig:2CP}
\end{figure}

\noindent
An $\AdSS$ crooked plane centered at the identity element is determined by a geodesic line $\ell$ of~$\HH^2$ and a choice of orientation (left or right).
We denote by $\ACP(\ell)$ the \emph{left $\AdSS$ crooked plane centered at $e\in G_0$ associated with~$\ell$}, which is described explicitly as follows (see Figure~\ref{fig:2CP}):
\begin{itemize}
  \item the interior of the stem $\mathsf{S}(\ell)$ of $\ACP(\ell)$ is the set of elliptic elements $h\in G_0$ whose fixed point belongs to~$\ell$;
  \item the boundary of the stem $\mathsf{S}(\ell)$ consists of $\{e\}$ and of all parabolic elements $h\in G_0$ fixing one of the endpoints $[v^+],[v^-]$ of $\ell$ in $\partial_{\infty}\HH^2$;
  \item the wings of $\ACP(\ell)$ are $\mathsf{W}(v^+)$ and $\mathsf{W}(v^-)$, where $\mathsf{W}(v^+)$ is the set of hyperbolic elements $h\in G_0$ with attracting fixed point $[v^+]$,~and~simi\-larly for~$v^-$.
\end{itemize}
In other words, $\ACP(\ell)$ is the set of orientation-preserving isometries of~$\HH^2$ (\ie elements of $G_0\simeq\AdS$) with a nonrepelling fixed point in~$\overline{\ell}$, where $\overline{\ell}$ is the closure of $\ell$ in $\HH^2\cup\partial_{\infty}\HH^2$.
A general \emph{left $\AdSS$ crooked plane} is a set of the form $g_2\ACP(\ell)g_1^{-1}$ where $\ell$ is a geodesic line of~$\HH^2$ and $g_1,g_2\in G_0$.
The image of a left $\AdSS$ crooked plane under the orientation-reversing isometry $g\mapsto g^{-1}$ of $\AdS$ is called a \emph{right $\AdSS$ crooked plane}. 
The right $\AdSS$ crooked plane centered at $e\in G_0$ associated with~$\ell$ will be denoted $\ACP^*(\ell)$: it is the inverse of $\ACP(\ell)$, \ie the set of orientation-preserving isometries of $\HH^2$ with a nonattracting fixed point in~$\overline{\ell}$.
Note that $\ACP^*(\ell)$ is obtained from $\ACP(\ell)$ by replacing each wing (which is an open lightlike half-plane) with the interior of the complementary lightlike half-plane.

An $\AdSS$ crooked plane (unlike a timelike geodesic plane) divides $\AdS$ into two connected components.
Given a transverse orientation of~$\ell$, we denote by $\ACHS_+(\ell)$ (\resp $\ACHS_+^*(\ell)$) the connected component of $\AdS\smallsetminus\ACP(\ell)$ (\resp $\AdS \smallsetminus \ACP^*(\ell)$) consisting of nontrivial elements $g\in G_0$ with a nonrepelling (\resp nonattracting) fixed point in $(\HH^2\cup\partial_{\infty}\HH^2)\smallsetminus\overline{\ell}$ lying on the positive side of~$\overline{\ell}$.
We denote the closures in $\AdS$ by $\overline{\ACHS_+(\ell)}=\ACHS_+(\ell)\cup\ACP(\ell)$ and $\overline{\ACHS_+^*(\ell)}=\ACHS_+^*(\ell)\cup\ACP^*(\ell)$.

\begin{remark}\label{rem:conj-ACP}
For any $g\in G_0$ and any geodesic line $\ell$ of~$\HH^2$,
$$\ACP(g\cdot\ell) = g \, \ACP(\ell) \, g^{-1}.$$
If $g$ takes the transverse orientation of~$\ell$ to the transverse orientation of $g\cdot\ell$, then
$$\overline{\ACHS_+(g\cdot\ell)} = g \, \overline{\ACHS_+(\ell)} \, g^{-1}.$$
\end{remark}

\begin{remark}\label{rem:extremelines}
The set $\partial_{\infty}\ACHS_+(\ell)$ (\resp $\partial_{\infty}\ACHS_+^*(\ell)$) of accumulation points in $\partial_{\infty}\AdS$ of $\ACHS_+(\ell)$ (\resp $\ACHS_+^*(\ell)$) is the projectivization of the set of rank-one matrices whose image (\resp kernel) in~$\RR^2$ projects to a point of $\PP^1(\RR)\simeq\partial_{\infty}\HH^2$ lying on the positive side of~$\overline{\ell}$.
For any $(g_1,g_2)\in (G\times G)_+$,
\begin{eqnarray*}
\partial_{\infty}\big(g_2\ACHS_+(\ell)g_1^{-1}\big) = \partial_{\infty}\ACHS_+(g_2\cdot\ell),\\
\partial_{\infty}\big(g_2\ACHS_+^*(\ell)g_1^{-1}\big) = \partial_{\infty}\ACHS_+^*(g_1\cdot\ell).\,
\end{eqnarray*}
\end{remark}

We define the \emph{stem quadrant} $\SQ(\ell)\subset\AdS\simeq G_0$ of the transversely oriented geodesic~$\ell$ to be the set of hyperbolic elements whose axis is orthogonal to~$\ell$ and that translate towards the positive side of~$\ell$.
It is one of the two connected components of the complement of the stem $\mathsf{S}(\ell)$ in the timelike plane containing $\mathsf{S}(\ell)$.

\begin{remarks}\label{rem:Einstein}
\begin{enumerate}
  \item $\AdSS$ crooked planes were introduced in \cite{dgk-strips} as analogues of Drumm's crooked planes in~$\RR^{2,1}$; the latter had initially been defined in \cite{dru92, dg99}.
  More precisely, let us view the Minkow\-ski space $\RR^{2,1}$ as the Lie algebra $\g=\mathfrak{psl}_2(\RR)$ of~$G$ endowed with half its Killing form (as in \cite{gm00, dgk13, dgk-strips}), or equivalently as the set of \emph{Killing vector fields} on~$\HH^2$ (see \cite[\S\,4.1]{dgk-strips}).
  Following \cite{dgk-strips}, the \emph{left crooked plane centered at $0\in\g$ associated with~$\ell$}, denoted $\CP(\ell)$, is by definition the set of Killing fields on~$\HH^2$ with a nonrepelling fixed point in~$\overline{\ell}$.
  The left crooked planes in~$\RR^{2,1}$ are the sets of the form $\CP(\ell)+v$ where $\ell$ is a geodesic of~$\HH^2$ and $v\in\RR^{2,1}$.
  The left $\AdSS$ crooked plane $\ACP(\ell)\subset G_0$ defined above is just the exponential of $\CP(\ell)\subset\g$.
  For any transverse orientation of~$\ell$, the $\AdSS$ stem quadrant $\SQ(\ell)\subset G_0$ is the exponential of the Minkowski stem quadrant $\mathcal{SQ}(\ell)\subset\g$, which is by definition the set of hyperbolic Killing fields whose axis is orthogonal to~$\ell$ and that translate towards the positive side of~$\ell$ (this is also called the \emph{cone of allowable translations} in \cite{bcdg13}).
  \item By \cite{gol13}, the closures of both Minkowski crooked planes and $\AdSS$ crooked planes inside the $3$-dimensional \emph{Einstein universe} $\mathrm{Ein}^3$ are \emph{crooked surfaces}, which were introduced by Frances \cite{fra03}.
  We refer to \cite{cfl13} for a recent construction of examples of fundamental domains bounded by crooked surfaces in~$\mathrm{Ein}^3$.
\end{enumerate}
In this paper, Minkowski crooked planes will only be used in Appendix~\ref{app:convergence}.
\end{remarks}

%%%%%%%%%%%%%%%%%%%%%%%%%%%%%%%%%%%%%%%%%%%%%%%%%%%
\section{Disjointness criteria for $\AdSS$ crooked planes}\label{sec:disjoint}

We first establish complete disjointness criteria for $\AdSS$ crooked planes, both left (Theorem~\ref{thm:AdS-disj-left}) and right (Theorem~\ref{thm:AdS-disj-right}).
In Conditions~(3) below, distances between ideal points should be interpreted as distances between horoballs centered at these ideal points: see Observation~\ref{obs:ext-diff-d} and Remark~\ref{rem:concrete-Fgg'}.

\begin{theorem}\label{thm:AdS-disj-left}
Let $\ell,\ell'$ be geodesic lines of~$\HH^2$.
For $g,g'\in G_0$, the following conditions are equivalent:
\begin{enumerate}
  \item the left $\AdSS$ crooked planes $g\ACP(\ell)$ and $g'\ACP(\ell')$ are disjoint;
  \item there exists $K > 0$ such that $d(g\cdot x, g'\cdot x') - d(x,x') \geq K$ for all $x\in\ell$ and $x'\in\ell'$;
  \item for any endpoints $\xi$ of~$\ell$ and $\xi'$ of~$\ell'$, we have $g\cdot\xi\neq g'\cdot\xi'$ and $d(g\cdot\nolinebreak\xi, g'\cdot\nolinebreak\xi') - d(\xi, \xi') > 0$.
\end{enumerate}
If these conditions hold, then $g\cdot\ell$ and $g'\cdot\ell'$ have disjoint closures in $\HH^2\cup\nolinebreak\partial_{\infty}\HH^2$.

The closed left $\AdSS$ crooked \emph{half-spaces} $\overline{\ACHS_+(g\cdot\ell)}\,g = g\,\overline{\ACHS_+(\ell)}$ and\linebreak $\overline{\ACHS_+(g'\cdot\ell')}\,g' = g'\,\overline{\ACHS_+(\ell')}$ are disjoint if and only if (1), (2), (3) hold and $g\cdot\ell$ and $g'\cdot\ell'$ are transversely oriented away from each other.
\end{theorem}

(Here we endow $\ell$ and~$\ell'$ with the transverse orientations induced by those of $g\cdot\ell$ and $g'\cdot\ell'$ via $g$ and~$g'$; see Remark~\ref{rem:conj-ACP}.)

\begin{theorem}\label{thm:AdS-disj-right}
Let $\ell,\ell'$ be geodesic lines of~$\HH^2$.
For $g,g'\in G_0$, the following conditions are equivalent:
\begin{enumerate}
  \item the right $\AdSS$ crooked planes $g\ACP^*(\ell)$ and $g'\ACP^*(\ell')$ are disjoint;
  \item there exists $K > 0$ such that $d(g\cdot x, g'\cdot x') - d(x,x') \leq -K$ for all $x\in\ell$ and $x'\in\ell'$;
  \item for any endpoints $\xi$ of~$\ell$ and $\xi'$ of~$\ell'$, we have $\xi\neq\xi'$ and $d(g\cdot\xi, g'\cdot\nolinebreak\xi') - d(\xi,\xi') < 0$.
\end{enumerate}
If these conditions hold, then $\ell$ and~$\ell'$ have disjoint closures in $\HH^2\cup\nolinebreak\partial_{\infty}\HH^2$.

The closed right $\AdSS$ crooked \emph{half-spaces} $g\overline{\ACHS_+^*(\ell)}$ and $g'\overline{\ACHS_+^*(\ell')}$ are disjoint if and only if (1), (2), (3) hold and $\ell$ and~$\ell'$ are transversely oriented away from each other.
\end{theorem}

\begin{remark}\label{rem:disj-CP-nodisj-geod}
In Theorem~\ref{thm:AdS-disj-left}, the left $\AdSS$ crooked planes $g\ACP(\ell)$ and $g'\ACP(\ell')$ can be disjoint in $\AdS$ without $\ell$ and~$\ell'$ being disjoint in~$\HH^2$.
For instance, $\ACP(\ell)\cap g\ACP(\ell)=\emptyset$ for all $g\in\SQ(\ell)$: see \cite[Rem.\,8.4]{dgk-strips}.
Similarly, in Theorem~\ref{thm:AdS-disj-right}, the right $\AdSS$ crooked planes $g\ACP^*(\ell)$ and $g'\ACP^*(\ell')$ can be disjoint in $\AdS$ without $g\cdot\ell$ and $g'\cdot\ell'$ being disjoint in~$\HH^2$.
This should be contrasted with the Minkowski setting: two crooked planes in~$\RR^{2,1}$ always intersect if the geodesic lines of~$\HH^2$ corresponding to their stems intersect (see \cite[Prop.\,6.3]{bcdg13}).
\end{remark}

%%%%%%%%%%%%%%%%%%%%%%%%%%%%%%%%%%%
\subsection{Meaning of Condition~(3)}

In Condition~(3) of both theorems above, the expression $d(g\cdot\xi, g'\cdot\xi') - d(\xi,\xi')$ should be understood as $F_{g,g'}(\xi,\xi')$, where $F_{g,g'}$ is defined as follows.
We set $\overline{\HH^2}:=\HH^2\cup\partial_{\infty}\HH^2$.

\begin{observation}\label{obs:ext-diff-d}
For any $g,g'\in G_0$, the function
$$F_{g,g'} :\ (x,x') \,\longmapsto\, d(g\cdot x,g'\cdot x') - d(x,x')$$
on $\HH^2\times\HH^2$, with values in~$\RR$, extends continuously to a function $F_{g,g'}$ on
$$H_{g,g'}:=(\overline{\HH^2}\times\overline{\HH^2}) \smallsetminus \{(\xi,\xi') \in \partial_{\infty}\HH^2\times\partial_{\infty}\HH^2 ~|~ \xi=\xi'\text{ and }g\cdot \xi = g'\cdot \xi' \},$$
with values in $\RR\cup\{-\infty, +\infty\}$.
\end{observation}

\begin{proof}
Let $C(\HH^2)$ be the space of continuous functions from $\HH^2$ to~$\RR$, endowed with the topology of uniform convergence on compact sets.
Fix a basepoint $p_0\in\HH^2$.
For $p\in\HH^2$, let $b_{p_0,p}\in C(\HH^2)$ be given by
$$b_{p_0,p}(x) := d(p,x) - d(p,p_0)$$
for all $x\in\HH^2$.
This defines a map $b_{p_0} : \HH^2\to C(\HH^2)$.
It is well known (see \cite[Ch.\,II, \S\,1]{bal95} for instance) that $b_{p_0}$ is continuous and injective, and extends to a continuous and injective map $b_{p_0} : \overline{\HH^2}\to C(\HH^2)$ sending any $\xi\in\partial_{\infty}\HH^2$ to the Busemann function associated with $\xi$ and~$p_0$.
Define
$$\begin{array}{rccc}
\overline{d} : & \overline{\HH^2}\times\overline{\HH^2} & \longrightarrow & \RR\cup\{-\infty\} \\
& (p,p') & \longmapsto & \inf_{\HH^2}(b_{p_0,p} + b_{p_0,p'}) .
\end{array}$$
Note that $\overline{d}(p,p')=-\infty$ if and only if $p=p'\in\partial_{\infty}\HH^2$; otherwise, $b_{p_0,p}+b_{p_0,p'}$ achieves its infimum $\overline{d}(p,p')\in\RR$ exactly on the geodesic segment (or ray,~or line) of~$\HH^2$ connecting $p$ to~$p'$.
Moreover, $\overline{d}(p,p')=d(p,p')-\nolinebreak d(p,p_0)-\nolinebreak d(p',p_0)$ for $p,p'\in\HH^2$.
The map $\overline{d}$ is clearly upper semicontinuous: if $(p_n,p'_n)\to (p,p')$ in $\overline{\HH^2}\times\overline{\HH^2}$, then $\limsup_n\,\overline{d}(p_n,p'_n) \leq \overline{d}(p,p')$.
It is also lower semicontinuous: indeed, suppose $\overline{d}(p,p')\neq -\infty$.
For large~$n$, the geodesic segments (or rays, or lines) from $p_n$ to~$p'_n$ (on which $b_{p_0,p_n}+b_{p_0,p'_n}=\overline{d}(p_n,p'_n)$) all visit a fixed compact region of~$\HH^2$; we then use the uniform convergence of $b_{p_0,p_n}+b_{p_0,p'_n}$ to $b_{p_0,p}+b_{p_0,p'}$ on this compact region.
Thus $\overline{d}$ is continuous on $\overline{\HH^2}\times\overline{\HH^2}$.

Similarly, for $g,g'\in G_0$, the map
$$\begin{array}{rccc}
\overline{d}_{g,g'} : & \overline{\HH^2}\times\overline{\HH^2} & \longrightarrow & \RR\cup\{-\infty\} \\
& (p,p') & \longmapsto & \inf_{\HH^2}(b_{g\cdot p_0,p} + b_{g'\cdot p_0,p'})
\end{array}$$
is continuous.
The function $F_{g,g'} : \HH^2\times\HH^2\to\RR$ extends continuously to a function $F_{g,g'} : H_{g,g'}\to\RR\cup\{ \pm\infty\}$ by setting
$$F_{g,g'}(p,p') := \overline{d}_{g,g'}(g\cdot p,g'\cdot p') - \overline{d}(p,p') .$$
Indeed, the right-hand side never has the form $\infty-\infty$ if $(p,p')\in H_{g,g'}$.
\end{proof}

Concretely, we can express $F_{g,g'}$ in terms of distances between horoballs:

\begin{remark}\label{rem:concrete-Fgg'}
For any $\xi,\xi'\in\partial_{\infty}\HH^2$ with $(\xi,\xi')\in H_{g,g'}$, for any small enough horoballs $B,B'$ of~$\HH^2$ centered respectively at $\xi,\xi'$, and for any $x,x'\in\HH^2$,
$$\left \{ \begin{array}{lll}
F_{g,g'}(\xi,x') & = & d(g\cdot B,g'\cdot x') - d(B,x'),\\
F_{g,g'}(x,\xi') & = & d(g\cdot x,g'\cdot B') - d(x,B'),\\
F_{g,g'}(\xi,\xi') & = & d(g\cdot B, g'\cdot B') - d(B, B'),
\end{array} \right .$$
where we define the distance between two disjoint subsets of~$\HH^2$ to be the infimum of distances between pairs of points in these two subsets, and declare the distance between two concentric horoballs to be $-\infty$.
\end{remark}

%%%%%%%%%%%%%%%%%%%%%%%%%%%%%%%%%%%
\subsection{Proof of Theorems \ref{thm:AdS-disj-left} and~\ref{thm:AdS-disj-right}}

By Remark~\ref{rem:conj-ACP}, the left $\AdSS$ crooked half-space $g\,\overline{\ACHS_+(\ell)}$, bounded by $g\ACP(\ell)$, is the inverse of the right $\AdSS$ crooked half-space $g^{-1}\,\overline{\ACHS^*_+(g\cdot\ell)}$, bounded by $g^{-1}\,\ACP^*(g\cdot\ell)$, and similarly for $g'\,\overline{\ACHS_+(\ell')}$ and ${g'}^{-1}\,\overline{\ACHS^*_+(g'\cdot\ell')}$.
Therefore, Theorem~\ref{thm:AdS-disj-left} is equivalent to Theorem~\ref{thm:AdS-disj-right}.
We now prove Theorem~\ref{thm:AdS-disj-right}.

\begin{proof}[Proof of Theorem~\ref{thm:AdS-disj-right}]
We use the notation $F_{g,g'}$ from Observation~\ref{obs:ext-diff-d}, and write $\overline{\ell}, \overline{\ell'}$ for the respective closures of $\ell, \ell'$ in $\overline{\HH^2}=\HH^2\cup\partial_{\infty}\HH^2$.

We first observe that if (2) holds, then $\overline{\ell}$ and~$\overline{\ell'}$ are disjoint in $\overline{\HH^2}$: otherwise, by taking $x\in\ell$ and $x'\in\nolinebreak\ell'$ close enough to each other, we would obtain $d(g\cdot x,g'\cdot x')=F_{g,g'}(x,x')+d(x,x')<0$, a contradiction.
In particular, if (2) or~(3) holds, then $F_{g,g'}$ is well defined and continuous on $\overline{\ell}\times\overline{\ell'}$ (Observation~\ref{obs:ext-diff-d}).

The implication $(2)\Rightarrow (3)$ follows from the continuity of~$F_{g,g'}$.

To prove $(3)\Rightarrow (2)$, suppose that $\ell$ and~$\ell'$ do not share any common endpoint at infinity and that $F_{g,g'}(x,x') \geq 0$ for some $(x,x')\in \overline{\ell}\times \overline{\ell'}$.
We claim that $x$ and $x'$ can both be taken in $\partial_{\infty}\HH^2$.
Indeed, given $x\in\nolinebreak\overline{\ell}$, let $D_x \subset \overline{\HH^2}$ be the set where $F_{g,g'}(x,\cdot)$ is defined (namely, $\overline{\HH^2}$ minus at most a singleton).
Consider the subset of points $x'\in D_x$ such that $F_{g,g'}(x,x')=d(x',g'^{-1}g\cdot x)-d(x',x)\geq 0$.
If nonempty, this subset is either the closure in~$D_x$ of a half-plane of~$\HH^2$ (if $x\neq g'^{-1}g\cdot x$) or the whole of~$D_x$ (otherwise). 
In particular, if this subset intersects $\overline{\ell'}\subset D_x$ then it contains an endpoint of~$\overline{\ell'}$.
Thus we may assume that $x'$ is ideal, equal to that endpoint.
Then, switching the roles of $x$ and~$x'$, we may assume that $x$ is ideal as well.
This proves $(3)\Rightarrow (2)$, by contraposition.

To prove $(2),(3)\Rightarrow (1)$, suppose that $\ell$ and~$\ell'$ do not share any common endpoint at infinity and that $F_{g,g'}\leq -K<0$ on $\overline{\ell}\times \overline{\ell'}$.
Consider $h\in\ACP^*(\ell)$ and $h'\in\ACP^*(\ell')$.
Recall (see Section~\ref{subsec:ACP}) that $\ACP^*(\ell)$ is the set of orientation-preserving isometries of~$\HH^2$ with a nonattracting fixed point in~$\overline{\ell}$.
Therefore, we can find a set $x\subset\HH^2$ that is either a singleton of~$\ell$ fixed by~$h$, or a small horoball centered in $\overline{\ell}$ that is preserved or expanded by~$h$.
Similarly, we can find a set $x'\subset\HH^2$ that is either a singleton of~$\ell'$ fixed by~$h'$, or a small horoball centered in $\overline{\ell'}$ that is preserved or expanded by~$h'$.
Since $x\subset h\cdot x$ and $x'\subset h'\cdot x'$, and since $F_{g,g'}<0$ on $\overline{\ell}\times\overline{\ell'}$, if $x$ and~$x'$ are taken small enough, then
$$d(gh\cdot x, g'h'\cdot x') = d(g\cdot x,g' \cdot x') -\lambda(h)-\lambda(h') < d(x,x') .$$
(As in Remark~\ref{rem:concrete-Fgg'}, we define the distance between two disjoint subsets of~$\HH^2$ to be the infimum of distances between pairs of points in these two subsets, and declare the distance between two concentric horoballs to be $-\infty$.
The function $\lambda$ is defined in \eqref{eqn:def-lambda}.)
In particular, $gh\neq g'h'$.
This proves $g\ACP^*(\ell) \cap g'\ACP^*(\ell')=\nolinebreak\emptyset$.

We observe that if $(1)$ holds, then $\ell$ and~$\ell'$ do not share any endpoint at infinity.
Indeed, by contraposition, suppose that $\ell$ and~$\ell'$ share an endpoint $\xi\in\partial_{\infty}\HH^2$.
Let $B$ be a horoball centered at~$\xi$.
We consider two cases: if $g\cdot \xi = g'\cdot \xi$, then for any hyperbolic isometry $h\in \ACP^*(\ell)$ that fixes~$\xi$ and expands~$B$ enough, the isometry $g'^{-1}gh$ also fixes~$\xi$ and expands~$B$, hence belongs to $\ACP^*(\ell')$, yielding $gh\in g\ACP^*(\ell) \cap g'\ACP^*(\ell')\neq\emptyset$.
If $g\cdot \xi \neq g'\cdot \xi$, choose $p\in\ell$ deep enough inside~$B$, so that $q:=g'^{-1}g\cdot p$ lies outside~$B$ (close to $g'^{-1}g\cdot \xi\neq \xi$).
There exists a rotation $h\in \ACP^*(\ell)$ centered at~$p$ such that $g'^{-1}gh\cdot \xi=\xi$.
The element $g'^{-1}gh$, which fixes~$\xi$ and takes $p$ to~$q$, expands~$B$, hence belongs to $\ACP^*(\ell')$.
Again we find $gh\in g\ACP^*(\ell) \cap g'\ACP^*(\ell')\neq\emptyset$.

We now prove $(1)\Rightarrow (3)$ by contraposition.
Suppose that $F_{g,g'}(\xi, \xi')\geq 0$ for some distinct endpoints $\xi$ of~$\ell$ and $\xi'$ of~$\ell'$; in particular, $g\cdot\xi\neq g'\cdot\xi'$ by definition of~$F_{g,g'}$ (see Remark~\ref{rem:concrete-Fgg'}).
Let $B,B'$ be two small horoballs centered respectively at $\xi, \xi'$.
The inequality $F_{g,g'}(\xi, \xi')\geq 0$ means that $g\cdot B$ is further away from $g'\cdot B'$ than $B$ is from $B'$. Therefore, among all orientation-preserving isometries taking the oriented line $(\xi, \xi')$ to $(g\cdot \xi, g'\cdot \xi')$, we can find one, let us call it $h$, such that $h\cdot B \supset g \cdot B$ and $h\cdot B' \supset g' \cdot B'$.
As a consequence, $g^{-1}h\in \ACP^*(\ell)$ and ${g'}^{-1}h\in \ACP^*(\ell')$, hence $h\in g\ACP^*(\ell) \cap g'\ACP^*(\ell')\neq\emptyset$.
This completes the proof of $(1)\Rightarrow (3)$.

Suppose that (1), (2), (3) hold.
The disjoint crooked planes $g\ACP^*(\ell)$ and $g'\ACP^*(\ell')$ define disjoint closed crooked half-spaces.
In order to make these the \emph{positive} ones $g\ACHS_+^*(\ell)$ and $g'\ACHS_+^*(\ell')$, we must transversely orient $\ell, \ell'$ \emph{away} from each other.
Indeed, by Remark~\ref{rem:extremelines}, the set of accumulation points in $\partial_{\infty}\AdS$ of $g\ACHS_+^*(\ell)$ (\resp $g'\ACHS_+^*(\ell')$) is the projectivization of the set of rank-one matrices whose projectivized kernel lies on the positive side of $\overline{\ell}$ (\resp of~$\overline{\ell'}$); the positive half-spaces of $\ell$ and~$\ell'$ in~$\overline{\HH^2}$ must thus be~disjoint.
\end{proof}

%%%%%%%%%%%%%%%%%%%%%%%%%%%%%%%%%%%%%%%%%%%%%%%%%%%
\section{Crooked fundamental domains and local isometries of~$\HH^2$}

In this section we give a necessary and sufficient condition for the existence of a fundamental domain in $\AdS$ bounded by pairwise disjoint $\AdSS$ crooked planes.
We may restrict to right crooked planes by the following remark.

\begin{remark}\label{rem:left-right-crooked-fund-dom}
Let $j,\rho\in\Hom(\Gamma,G_0)$ be two representations of a discrete group~$\Gamma$.
The group $\Gamma^{j,\rho}$ (see \eqref{eqn:def-jr}) acts properly discontinuously on $\AdS$ and admits a fundamental domain bounded by \emph{right} $\AdSS$ crooked planes if and only if the group $\Gamma^{\rho,j}$ acts properly discontinuously on $\AdS$ and admits a fundamental domain bounded by \emph{left} $\AdSS$ crooked planes.
\end{remark}

Indeed, switching $j$ and~$\rho$ amounts to conjugating the action on $\AdS=G_0$ by the orientation-reversing isometry $g\mapsto g^{-1}$, which switches left and right $\AdSS$ crooked planes.

%%%%%%%%%%%%%%%%%%%%%%%%%%%%%%%%%%%
\subsection{A necessary and sufficient condition}

Here is our main result.

\begin{theorem}\label{thm:charact-crooked-domains}
Let $\Gamma$ be a finitely generated free group and $j,\rho\in\Hom(\Gamma,G)$ two representations.
The group $\Gamma^{j,\rho}$ acts properly discontinuously on $\AdS$ and admits a fundamental domain in $\AdS$ bounded by pairwise disjoint right $\AdSS$ crooked planes if and only if there exist
\begin{itemize}
  \item a $j(\Gamma)$-invariant collection $\A$ of geodesic lines of~$\HH^2$ with pairwise disjoint closures in $\HH^2\cup\partial_{\infty}\HH^2$, dividing $\HH^2$ into fundamental domains for $j(\Gamma)$,
  \item a $(j,\rho)$-equivariant map $f : \bigcup_{\alpha\in\A}\alpha\rightarrow\HH^2$ which is an \emph{isometry} in restriction to any $\alpha\in\A$,
  \item a constant $K>0$ such that for any $\alpha\neq\alpha'$ in~$\A$, any $x\in\alpha$, and any $x'\in\alpha'$,
  \begin{equation}\label{eqn:f-brings-closer}
  d(f(x),f(x')) \leq d(x,x') - K .
  \end{equation}
\end{itemize}
In this case, $j$ is convex cocompact and $\rho$ is uniformly shorter than~$j$, in the sense of \eqref{eqn:propcritAdS}; a collection of pairwise disjoint right $\AdSS$ crooked planes dividing $\AdS$ into fundamental domains for~$\Gamma^{j,\rho}$ is given by the $g_{\alpha}\ACP^*(\alpha)$ for $\alpha\in\A$, where $g_{\alpha}\in G_0$ is the unique orientation-preserving isometry of~$\HH^2$ whose restriction to~$\alpha$ coincides with~$f$.
\end{theorem}

\begin{remark}\label{rem:adjacent-edges}
By the triangle inequality, it is sufficient for \eqref{eqn:f-brings-closer} to hold for any \emph{adjacent} $\alpha\neq\alpha'$ in~$\A$.
\end{remark}

\begin{proof}[Proof of Theorem~\ref{thm:charact-crooked-domains}]
Suppose that $\Gamma^{j,\rho}$ acts properly discontinuously on $\AdS$ and admits a fundamental domain in $\AdS$ bounded by pairwise disjoint right $\AdSS$ crooked planes.
Translates of this fundamental domain define a $\Gamma^{j,\rho}$-invariant collection $(\DD^*_{\alpha})_{\alpha\in\A}$ of pairwise disjoint $\AdSS$ crooked planes in~$\AdS$.
For any $\alpha \in \A$, write $\DD^*_{\alpha}=g_{\alpha}\ACP^*(\ell_\alpha)$, where $g_{\alpha}\in G_0$ and $\ell_\alpha$ is a geodesic line of~$\HH^2$.
Since the $\DD^*_{\alpha}$ are pairwise disjoint, Theorem~\ref{thm:AdS-disj-right} implies that the $\ell_{\alpha}$ are distinct, and so we may identify $\A$ with a collection of geodesic lines of~$\HH^2$, writing $\alpha$ for~$\ell_{\alpha}$.
In fact, Theorem~\ref{thm:AdS-disj-right} implies that the closures in $\HH^2\cup\partial_{\infty}\HH^2$ of the geodesics $\alpha\in\A$ are pairwise disjoint.
By Remark~\ref{rem:conj-ACP}, for any $\alpha\in\A$ and $\gamma\in\Gamma$,
\begin{equation}\label{eqn:f-equiv}
\rho(\gamma) \DD^*_{\alpha} \, j(\gamma)^{-1} = \rho(\gamma) g_{\alpha} \ACP^*(\ell_{\alpha}) j(\gamma)^{-1} = \rho(\gamma) g_{\alpha} j(\gamma)^{-1} \, \ACP^*\big(j(\gamma)\cdot\alpha\big) ,
\end{equation}
and so the collection $\A$ is $j(\Gamma)$-invariant.
In particular, $j$ is convex cocompact (otherwise, lines exiting a cusp would lift to lines of~$\HH^2$ that meet at infinity), and $\A$ divides~$\HH^2$ into fundamental domains for the action of $j(\Gamma)$.
Let $f : \bigcup_{\alpha\in\A}\alpha\rightarrow\HH^2$ be the map whose restriction to any $\alpha\in\A$ coincides with~$g_{\alpha}$.
By \eqref{eqn:f-equiv}, we have $g_{j(\gamma)\cdot \alpha}=\rho(\gamma) g_{\alpha} j(\gamma)^{-1}$ for any $\gamma\in\Gamma$ and $\alpha\in\A$, hence $f$ is $(j,\rho)$-equivariant.
By Theorem~\ref{thm:AdS-disj-right}, for any $\alpha\neq\alpha'$ in~$\A$ there exists $K>0$ such that for any $x\in\alpha$ and $x'\in\alpha'$,
\begin{equation}\label{eqn:f-brings-closer-bis}
d(f(x),f(x')) = d(g_{\alpha}\cdot x,g_{\alpha'}\cdot x') \leq d(x,x') - K .
\end{equation}
By finiteness of $\A$ modulo $j(\Gamma)$, this $K$ can be taken to be uniform for any $\alpha\neq\alpha'$ that are adjacent to a common \emph{tile}, where a tile is by definition a connected component of $\HH^2\smallsetminus\bigcup_{\alpha\in\A}$.
To see uniformity over all $\alpha\neq\alpha'$, we apply the triangle inequality to the intersection points of $[x,x']$ with $\bigcup_{\alpha\in\A} \alpha$ and obtain
\begin{equation}\label{eqn:n-x-x'}
d(f(x),f(x')) \leq d(x,x') - n_{x,x'}\,K
\end{equation}
for all $x,x'\in\bigcup_{\alpha\in\A} \alpha$, where $n_{x,x'}\in\NN$ is the number of tiles meeting the geodesic segment $[x,x']$.
We now prove that $\rho$ is uniformly shorter than~$j$.
By compactness of the convex core $\Omega$ of $j(\Gamma)\backslash\HH^2$, each tile meets the preimage $\widetilde{\Omega}\subset\HH^2$ of~$\Omega$ in a compact set, of diameter bounded by some uniform constant $K_0>0$.
In particular, for any $x,x'\in\widetilde{\Omega}\cap\bigcup_{\alpha\in\A} \alpha$,
$$d(x,x') \leq n_{x,x'} K_0.$$
Using \eqref{eqn:n-x-x'}, we obtain
$$d(f(x),f(x')) \leq \bigg(1 - \frac{K}{K_0}\bigg) \, d(x,x')$$
for all $x,x'\in\widetilde{\Omega}\cap\bigcup_{\alpha\in\A} \alpha$.
Taking $x$ on the translation axis of $j(\gamma)$ and $x'=j(\gamma)\cdot x$, and using the $(j,\rho)$-equivariance of~$f$, we obtain $\lambda(\rho(\gamma))\leq (1-K/K_0)\,\lambda(j(\gamma))$ for all $\gamma\in\Gamma$.
This proves that $\rho$ is uniformly shorter than~$j$.

Conversely, suppose that there exist a $j(\Gamma)$-invariant collection $\A$ of pairwise disjoint geodesics of~$\HH^2$, a $(j,\rho)$-equivariant map $f : j(\Gamma)\cdot\bigcup_{\alpha\in\A}\alpha\rightarrow\HH^2$ whose restriction to any $\alpha\in\A$ coincides with some isometry $g_{\alpha}\in G_0$ of~$\HH^2$, and a constant $K>0$ such that for any $\alpha\neq\alpha'$ in~$\A$ and $(x,x')\in\alpha\times\alpha'$,
$$d(f(x),f(x')) \leq d(x,x') - K .$$
By Theorem~\ref{thm:AdS-disj-right}, the $\AdSS$ crooked planes $\DD^*_{\alpha}:=g_{\alpha}\ACP^*(\alpha)$, for $\alpha\in\A$, are pairwise disjoint.
Since $f$ is $(j,\rho)$-equivariant, $\DD^*_{j(\gamma)\cdot\alpha}=\rho(\gamma)\DD^*_{\alpha}j(\gamma)^{-1}$ for all $\gamma\in\Gamma$ and $\alpha\in\A$ (see \eqref{eqn:f-equiv}).
Let $\alpha_1,\dots,\alpha_m\in\A$ bound a fundamental domain of $\HH^2$ for the action of $j(\Gamma)$, and let $\mathsf{R}$ be the closed subset of $\AdS$ bounded by $\DD^*_{\alpha_1},\dots,\DD^*_{\alpha_m}$.
Then $\Gamma^{j,\rho}$ acts properly discontinuously on $\Gamma^{j,\rho}\cdot\mathsf{R}$, and $\mathsf{R}$ is a fundamental domain for this action.
To prove that $\Gamma^{j,\rho}\cdot\mathsf{R}$ is equal to the whole of $\AdS$, it is enough to show that $\DD^*_{\alpha}$ goes to infinity in $\AdS$ as $\alpha\in\A$ goes to infinity in~$\HH^2$.
Fix a point $x_0$ on some geodesic line $\alpha_0\in\A$.
Consider an element $\alpha\in\A$ and a point
$$g_{\alpha}h\in\DD^*_{\alpha},$$
where $h\in\ACP^*(\alpha)$.
By definition of $\ACP^*(\alpha)$ (see Section~\ref{subsec:ACP}), we can find a subset $x$ of~$\HH^2$ that is either a singleton of~$\alpha$ fixed by~$h$, or a small horoball centered at an endpoint of~$\alpha$ that is preserved or expanded by~$h$.
By the triangle inequality and \eqref{eqn:n-x-x'} (which remains true when $x$ is a small horoball, due to Observation~\ref{obs:ext-diff-d} and Remark~\ref{rem:concrete-Fgg'}),
\begin{eqnarray*}
d(x_0,x) & = & d(g_{\alpha}h\cdot x_0, g_{\alpha}h\cdot x)\\
& \leq & d(g_{\alpha}h\cdot x_0, g_\alpha \cdot x) \ =\ d(g_{\alpha}h\cdot x_0, f(x))\\ 
& \leq & d(g_{\alpha}h\cdot x_0, x_0) + d(x_0, f(x_0)) + d(f(x_0), f(x))\\
& \leq & d(g_{\alpha}h\cdot x_0, x_0) + d(x_0, f(x_0)) + d(x_0, x) - n_{x_0,x} K ,
\end{eqnarray*}
and so $d(g_{\alpha}h\cdot x_0,x_0)\geq n_{x_0,x} K - d(x_0, f(x_0))$.
Note that $n_{x_0,x}\in\NN$ depends only on~$\alpha$, and goes to infinity as $\alpha\in\A$ goes to infinity in~$\HH^2$.
Therefore, $\DD^*_{\alpha}$ goes to infinity in $\AdS$ as $\alpha\in\A$ goes to infinity in~$\HH^2$.
\end{proof}

%%%%%%%%%%%%%%%%%%%%%%%%%%%%%%%%%%%
\subsection{Disjoint $\AdSS$ crooked planes lead to proper actions on $\AdS$}

In \cite{dru92}, Drumm proved that gluing together a polyhedron $\mathcal{R}$ bounded by pairwise disjoint crooked planes in the $3$-dimensional Minkowski space $\RR^{2,1}$ yields a \emph{Lorentzian Schottky group}, \ie a free group acting properly discontinuously with $\mathcal{R}$ as a fundamental domain.
One direction of Theorem~\ref{thm:charact-crooked-domains} immediately implies the following analogue of Drumm's theorem in the $\AdSS$ setting.
(We denote by $\mathrm{Int}$ the interior of a subset of~$\HH^2$.)

\begin{corollary}\label{cor:AdS-Drumm-shortening}
Let $\Gamma$ be a finitely generated free group with free generating subset $\{ \gamma_1,\dots,\gamma_r\}$.
Let $L_1, L_1', \ldots, L_r, L_r'$ be closed half-planes in~$\HH^2$ with pairwise disjoint closures in $\HH^2\cup\partial_{\infty}\HH^2$, and let $j\in\Hom(\Gamma,G_0)$ be a representation such that
$$j(\gamma_i)\cdot L_i = \HH^2 \smallsetminus \mathrm{Int}(L_i')$$
for all~$i$.
Suppose that for some $g_i,g'_i\in G_0$, the $2r$ right $\AdSS$ crooked half-spaces $g_i\overline{\ACHS^*_+(\ell_i)}$ and $g_i' \overline{\ACHS^*_+(\ell_i')}$, for $1\leq i\leq r$, are all pairwise disjoint, where $\ell_i, \ell_i'$ are the transversely oriented geodesic lines of~$\HH^2$ whose positive half-planes are $L_i, L_i'$, respectively.
Then the representation $\rho\in\Hom(\Gamma,G_0)$ defined by
$$\rho(\gamma_i) = g'_i \, j(\gamma_i) \, g_i^{-1}$$
for all~$i$ is uniformly shorter than~$j$ in the sense of \eqref{eqn:propcritAdS}; the group $\Gamma^{j,\rho}$ acts properly discontinuously on $\AdS$, and a fundamental domain $\mathsf{R}$ is given by the complement in $\AdS$ of the union of the $\AdSS$ crooked half-spaces $g_i\overline{ \ACHS^*_+(\ell_i)}$ and $g_i' \overline{\ACHS^*_+(\ell_i')}$.
\end{corollary}

\begin{proof}
Let $\A$ be the set of $j(\Gamma)$-translates of the~$\ell_i$, for $1\leq i\leq r$, and let $f : \bigcup_{\alpha\in\A} \alpha\rightarrow\HH^2$ be the $(j,\rho)$-equivariant map whose restriction to $\ell_i$ coincides with~$g_i$ for all~$i$ (in particular, the restriction of $f$ to $\ell'_i$ coincides with~$g'_i$).
By Theorem~\ref{thm:AdS-disj-right}, there exists $K>0$ such that \eqref{eqn:f-brings-closer} holds for any $\alpha\neq\alpha'$ in $\{ \ell_i,\ell'_i~|~1\leq i\leq r\}$, any $x\in\alpha$, and any $x'\in\alpha'$.
By equivariance, this is still true for any adjacent $\alpha\neq\alpha'$ in~$\A$, any $x\in\alpha$, and any $x'\in\alpha'$.
We conclude using Theorem~\ref{thm:charact-crooked-domains} and Remark~\ref{rem:adjacent-edges}.
\end{proof}

%%%%%%%%%%%%%%%%%%%%%%%%%%%%%%%%%%%%%%%%%%%%%%%%%%%
\section{Proper actions with no crooked fundamental domain}\label{sec:no-crooked-fd}

In this section we use the other direction of Theorem~\ref{thm:charact-crooked-domains} to prove Theorem~\ref{thm:no-crooked-fd}, which states the existence of proper actions without any crooked fundamental domain in $\AdS$.
We provide three families of examples.

%%%%%%%%%%%%%%%%%%%%%%%%%%%%%%%%%%%
\subsection{Necessity of convex cocompactness}

Here is an immediate consequence of Remark~\ref{rem:left-right-crooked-fund-dom} and Theorem~\ref{thm:charact-crooked-domains} --- more precisely, of the fact that in the conclusion of Theorem~\ref{thm:charact-crooked-domains} the representation $j$ is convex cocompact.

\begin{proposition}\label{prop:cusp-no-crook}
Let $\Gamma$ be a finitely generated free group and let $j,\rho\in\Hom(\Gamma,G_0)$ be two representations such that the group $\Gamma^{j,\rho}$ acts properly discontinuously on $\AdS$.
If neither $j$ nor~$\rho$ is convex cocompact, then the group $\Gamma^{j,\rho}$ cannot admit a fundamental domain in $\AdS$ bounded by pairwise disjoint $\AdSS$ crooked planes.
\end{proposition}

More precisely, if $\Gamma^{j,\rho}$ admits a fundamental domain bounded by right (\resp left) $\AdSS$ crooked planes, then $j$ (\resp $\rho$) needs to be convex cocompact and $\rho$ (\resp $j$) uniformly shorter than $j$ (\resp $\rho$) in the sense of~\eqref{eqn:propcritAdS}.

Thus for any injective and discrete representation $j\in\Hom(\Gamma,G_0)$ such that the group $j(\Gamma)$ has a parabolic element, and for any uniformly shorter representation $\rho\in\Hom(\Gamma,G_0)$ in the sense of \eqref{eqn:propcritAdS} (for instance the constant representation), the group $\Gamma^{j,\rho}$ acts properly discontinuously on $\AdS$ (see Section~\ref{subsec:prop-AdS}) but does not admit any crooked fundamental domain.

Note that the analogue of Proposition~\ref{prop:cusp-no-crook} in the Minkowski space~$\RR^{2,1}$ is \emph{not} true: for any injective and discrete representation $j\in\Hom(\Gamma,G_0)$ of the finitely generated free group~$\Gamma$ (even if $j(\Gamma)$ has parabolic elements), the set of $j$-cocycles $u : \Gamma\rightarrow\g$ such that $\Gamma^{j,u}:=\{ (j(\gamma),u(\gamma))\,|\,\gamma\in\Gamma\}\subset G_0\ltimes\g$ acts properly discontinuously on $\RR^{2,1}\simeq\g$ by affine transformations is nonempty, and the corresponding groups $\Gamma^{j,u}$ all admit fundamental domains in~$\RR^{2,1}$ bounded by pairwise disjoint crooked planes \cite{dgk-parab}.

%%%%%%%%%%%%%%%%%%%%%%%%%%%%%%%%%%
\subsection{Examples with $\rho$ constant}

Let us now only consider proper actions of $\Gamma^{j,\rho}$ with $j$ convex cocompact.

\begin{proposition}\label{prop:nocrook1}
Let $\Gamma$ be a free group and $j\in\Hom(\Gamma,G_0)$ a convex cocompact representation such that the convex core $\Omega$ of the quotient surface $S = j(\Gamma)\backslash \HH^2$ has only one boundary component, of length $D>0$. 
Let $\rho\in\Hom(\Gamma,G_0)$ be the constant representation.
If the action of $\Gamma^{j,\rho}$ on $\AdS$ admits a fundamental domain bounded by pairwise disjoint $\AdSS$ crooked planes, then the action of $j(\Gamma)$ on~$\HH^2$ admits a fundamental domain $\mathscr{F}$ bounded by pairwise disjoint geodesic lines, each meeting the preimage $\widetilde{\Omega}\subset\HH^2$ of~$\Omega$ along segments of length $\leq D/2$.
\end{proposition}

\begin{proof}
Suppose the action of $\Gamma^{j,\rho}$ on $\AdS$ admits a fundamental domain bounded by disjoint $\AdSS$ crooked planes.
By Remark~\ref{rem:left-right-crooked-fund-dom} and Theorem~\ref{thm:charact-crooked-domains}, we can find pairwise disjoint geodesic lines $\ell_1,\dots, \ell_r,\ell'_1, \dots, \ell'_r$ in~$\HH^2$, bounding a fundamental domain $\mathscr{F}$ for the action of $j(\Gamma)$, and a $(j,\rho)$-equivariant, $1$-Lipschitz map $f : j(\Gamma)\cdot \bigcup_{i=1}^r \ell_i \rightarrow \HH^2$ whose restriction to any~$\ell_i$ is an isometry.
For $1\leq i\leq r$, let $p_i, q_i$ be the two intersection points of~$\ell_i$ with the boundary $\partial\widetilde{\Omega}$ of the convex core $\widetilde{\Omega}\subset\HH^2$ (see Figure~\ref{fig:proof-B3}).
Since $\rho$ is constant, $f(j(\Gamma)\cdot p_i)$ is a single point $P_i$ and $f(j(\Gamma)\cdot q_i)$ a single point $Q_i$. 
We have $d(P_i,Q_i)\leq D/2$ because some $j(\gamma)\cdot q_i$ is $D/2$-close to $p_i$ along $\partial\widetilde{\Omega}$ and $f$ is $1$-Lipschitz.
But $d(P_i, Q_i)=d(p_i,q_i)$ is also the length inside $\widetilde{\Omega}$ of an arc bounding~$\mathscr{F}$.
\end{proof}

\begin{figure}[ht!]
\centering
\labellist
\small\hair 2pt
\pinlabel $q_{2}$ [r] at 160 110
\pinlabel $p_{1}$ [u] at 115 145
\pinlabel $p_{2}'$ [d] at 21 64
\pinlabel $q_{1}'$ [d] at 59 27
\pinlabel $p_{1}'$ [d] at 115 27
\pinlabel $p_{2}$ [d] at 150 65
\pinlabel $q_{1}$ [u] at 55 145
\pinlabel $q_{2}'$ [l] at 12 110
\pinlabel $Q_2$ [l] at 317 140
\pinlabel $P_1$ [l] at 375 80
\pinlabel $P_2$ [l] at 340 25
\pinlabel $Q_1$ [r] at 285 80
\pinlabel $\overset{f}{\longrightarrow}$ [c] at 207 90
\endlabellist
\includegraphics[scale=0.7]{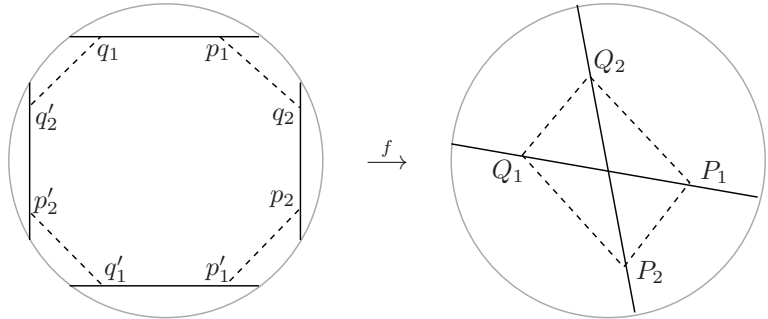}
\caption{An illustration of the proof of Proposition~\ref{prop:nocrook1} when $S$ is a one-holed torus. The boundary of the convex core $\widetilde{\Omega}\subset\HH^2$ is dashed. We set $p'_i:=j(\gamma_i)\cdot p_i$ and $q'_i:=j(\gamma_i)\cdot q_i$, as well as $P_i  := f(p_i) = f(p_i')$ and $Q_i := f(q_i) = f(q_i')$.}
\label{fig:proof-B3}
\end{figure} 

Proposition~\ref{prop:nocrook1} implies Theorem~\ref{thm:no-crooked-fd}.
Indeed, Randol's Collar Lemma \cite{ran79,bus78} states that in a hyperbolic surface~$S$, a geodesic boundary component of length~$D$ always has an embedded uniform neighborhood of radius $\psi(D)$ for some universal function $\psi : \RR_+\rightarrow\RR_+$ with $\lim_0 \psi=+\infty$.
Any arc from the boundary component to itself must have length $\geq\psi(D)$, hence $>D/2$ if $D$ is small enough.
Suppose $S$ is convex cocompact with only one boundary component, of such small length~$D$.
By Proposition~\ref{prop:nocrook1}, if $j\in\Hom(\Gamma,G_0)$ is the holonomy representation of~$S$ and $\rho\in\Hom(\Gamma,G_0)$ the constant homomorphism, then the group $\Gamma^{j,\rho}$ does not admit a fundamental domain in $\AdS$ bounded by pairwise disjoint $\AdSS$ crooked planes.

%%%%%%%%%%%%%%%%%%%%%%%%%%%%%%%%%%
\subsection{An open set of examples}

In a related vein, here are examples where $j$ is convex cocompact and the image of $\rho$ contains an elliptic element.

\begin{proposition}\label{prop:nocrook2}
Let $\Gamma$ be a finitely generated free group.
For any $\theta\in (0,\pi]$ there exists $\delta>0$ with the following property: if $(j,\rho)\in\Hom(\Gamma,G_0)^2$ is a pair of representations with $j$ convex cocompact and $\rho$ uniformly shorter than~$j$ in the sense of \eqref{eqn:propcritAdS}, and if the group $\Gamma^{j,\rho}$ admits a fundamental domain bounded by pairwise disjoint $\AdSS$ crooked planes, then
$$\lambda(j(\gamma)) \geq \delta$$
for any $\gamma\in\Gamma$ such that $\rho(\gamma)$ is a rotation of angle $\geq \theta$ in absolute value.
\end{proposition}

\begin{proof}
Let $(j,\rho)\in\Hom(\Gamma,G_0)^2$ be a pair of representations with $j$ convex cocompact and $\rho$ uniformly shorter than~$j$.
By Remark~\ref{rem:left-right-crooked-fund-dom} and Theorem~\ref{thm:charact-crooked-domains}, if $\Gamma^{j,\rho}$ admits a fundamental domain bounded by pairwise disjoint $\AdSS$ crooked planes, then we can find a collection $\A$ of pairwise disjoint geodesic lines dividing~$\HH^2$ into fundamental domains for the action of $j(\Gamma)$, and a $(j,\rho)$-equivariant, $1$-Lipschitz map $f : \bigcup_{\alpha\in\A} \alpha\to\HH^2$ whose restriction to any $\alpha\in\A$ is an isometry.
For any nontrivial $\gamma\in\Gamma$, we can find $\alpha\in\A$ intersecting the translation axis of $j(\gamma)$.
When $\lambda(j(\gamma))$ is small, the set of points $p\in\alpha$ such that $d(p,j(\gamma)\cdot p)\leq 1$ is a segment $I\subset\alpha$ of length at least $2\, \vert \log \lambda(j(\gamma))\vert +O(1)$.
On the other hand, if $\rho(\gamma)$ is a rotation of angle~$\theta$, then the set of points $q\in\HH^2$ such that $d(q,\rho(\gamma)\cdot q)\leq 1$ is a ball centered at the fixed point of $\rho(\gamma)$, whose radius depends only on~$\theta$.
The map $f$ must take $I$, isometrically, into this ball: this gives a lower bound on $\lambda(j(\gamma))$. 
\end{proof}

Proposition~\ref{prop:nocrook2} also implies Theorem~\ref{thm:no-crooked-fd}.
Indeed, it is easy to construct pairs $(j,\rho)$ with $j$ convex cocompact and $\rho$ uniformly shorter than~$j$ such that there exists an element $\gamma\in\Gamma$ for which $j(\gamma)$ is a very short translation and $\rho(\gamma)$ a fairly large rotation.
For instance, consider a free generating subset $\{ \gamma_1,\dots,\gamma_r\}$ of~$\Gamma$ and let $j$ send the $\gamma_i$, for $i\geq 1$, to hyperbolic elements playing ping pong in~$\HH^2$, with $\lambda(j(\gamma_1))$ very small.
Let $\rho$ send the~$\gamma_i$, for $i\geq 1$, to elliptic elements with a common fixed point in~$\HH^2$, with the rotation angle of $\rho(\gamma_1)$ fairly large.
Then $\rho$ is uniformly shorter than~$j$, hence the group $\Gamma^{j,\rho}$ acts properly discontinuously on $\AdS$ (see Section~\ref{subsec:prop-AdS}).
However, $\Gamma^{j,\rho}$ does not admit any fundamental domain in $\AdS$ bounded by pairwise disjoint $\AdSS$ crooked planes, by Proposition~\ref{prop:nocrook2}.

Note that the existence of an element~$\gamma$ for which $j(\gamma)$ is a short translation and $\rho(\gamma)$ a large rotation is stable under small deformations of $(j,\rho)$.
So is the property that $\rho$ be uniformly shorter than~$j$ (see \cite{kasPhD}, \cite[Prop.\,1.5]{gk13}).
Thus the absence of fundamental domains bounded by $\AdSS$ crooked planes is not a local accident.
By deforming the examples above, we can construct examples with $\rho(\Gamma)$ Zariski-dense in~$G_0$.

%%%%%%%%%%%%%%%%%%%%%%%%%%%%%%%%%%
\subsection{A question}

To conclude, we ask the following question which, in light of \cite[Thm\,1.8]{dgk-strips} on the one hand, and of Theorem~\ref{thm:no-crooked-fd} on the other, seems natural and nontrivial.

\begin{question}\label{q:fund-domain}
Let $\Gamma$ be a finitely generated free group.
For which pairs $(j,\rho)\in\Hom(\Gamma,G_0)^2$ with $j$ injective and discrete, and with $(j,\rho)$ satisfying the uniform contraction property \eqref{eqn:propcritAdS}, does $\Gamma^{j,\rho}$ admit a fundamental domain in $\AdS$ bounded by (right) $\AdSS$ crooked planes?
\end{question}

Note that by using Theorem~\ref{thm:charact-crooked-domains} and analyzing the situation carefully as in the proof of Proposition~\ref{prop:nocrook1}, it is possible to construct convex cocompact representations~$j$ such that for $\rho$ constant (\ie $\rho\equiv 1$) the group $\Gamma^{j,\rho}$ \emph{does} admit a fundamental domain bounded by $\AdSS$ crooked planes.

%%%%%%%%%%%%%%%%%%%%%%%%%%%%%%%%%%%%%%%%%%%%%%%%%%%
%%%%%%%%%%%%%%%%%%%%%%%%%%%%%%%%%%%%%%%%%%%%%%%%%%
\appendix
\section{Convergence of crooked fundamental domains}\label{app:convergence}

In \cite{dgk13} we showed that any Margulis spacetime is a rescaled limit, in projective space, of collapsing $\AdSS$ $3$-manifolds.
Here we explain how the ideas of~\cite{dgk-strips}  can be applied to make crooked fundamental domains for such collapsing $\AdSS$ manifolds converge in projective space to a fundamental domain for the Margulis spacetime. Let us assume in this appendix, as in~\cite{dgk-strips}, that the linear holonomy of the Margulis spacetime is convex cocompact.

%%%%%%%%%%%%%%%%%%%%%%%%%%%%%%%%%%%
\subsection{Embedding $\AdS$ and~$\RR^{2,1}$ into projective space}\label{subsec:embedP3}

Here is a quick review of how a transition from $\AdS$ to $\RR^{2,1}$ may be realized in real projective geometry. See~\cite[\S 7.1]{dgk13} for an expanded treatment.

The map
$$I: \begin{pmatrix} y_1+y_4 & y_2-y_3 \\ y_2+y_3 & -y_1+y_4 \end{pmatrix} \longmapsto [y_1:y_2:y_3:y_4]$$
defines an embedding of $\AdS = G_0 = \PSL_2(\RR)$ into $\PP^3(\RR)$ whose image is the open set $\{ [y]\in\PP^3(\RR) \,|\, y_1^2 + y_2^2 - y_3^2 - y_4^2 < 0\}$ (the interior of a projective quadric); it induces an injective group homomorphism $I_{\ast} :\ \mathrm{Isom}(\AdS)_0 = G_0\times G_0\ \hookrightarrow\ \PGL_4(\RR)$, and $I$ is $I_*$-equivariant: $I(A\cdot x) = I_*(A)\cdot I(x)$ for all $A\in\mathrm{Isom}(\AdS)_0$ and $x\in\AdS$.
As in Remark~\ref{rem:Einstein}.(1), we view the Minkowski space $\RR^{2,1}$ as the Lie algebra $\g=\mathfrak{psl}_2(\RR)$ of~$G$ endowed with half its Killing form, or equivalently as the set of Killing vector fields on~$\HH^2$ (see \cite[\S\,4.1]{dgk-strips}).
The map 
$$i: \begin{pmatrix} z_1 & z_2-z_3 \\ z_2+z_3 & -z_1 \end{pmatrix} \longmapsto [z_1:z_2:z_3:1]$$
defines an embedding of $\RR^{2,1} = \g = \mathfrak{psl}_2(\RR)$ into $\PP^3(\RR)$ whose image is the affine chart $\{ [z]\in\PP^3(\RR) \,|\, z_4 \neq 0\}$\,; it induces an injective group homomorphism $i_{\ast} :\ \mathrm{Isom}(\RR^{2,1})_0 = G_0\ltimes\g\ \hookrightarrow\ \PGL_4(\RR)$, and $i$ is $i_{\ast}$-equivariant: $i\big(B\cdot w) = i_{\ast}(B) \cdot i(w)$ for all $B\in\mathrm{Isom}(\RR^{2,1})_0$ and $w\in\RR^{2,1}$.
For $t>0$, consider the projective transformation
$$\small
r_t := \begin{pmatrix} t^{-1} & & & \\ & t^{-1} & & \\ & & t^{-1} & \\ & & & 1\end{pmatrix} \in \PGL_4(\RR).$$
Then $r_t\cdot I(\AdS)\subset r_{t'}\cdot I(\AdS)$ for $0<t'<t$ and
$$\bigcup_{t>0}\ r_t \cdot I(\AdS) \,=\, i(\RR^{2,1}) \cup \HH^2_{\infty}\,,$$
where $\HH^2_{\infty}:=\{ [y]\in\PP^3(\RR)\,|\, y_1^2 + y_2^2 - y_3^2 < 0 = y_4\}$ is a copy of the hyperbolic plane.
The limit as $t \to 0$ of the action of $r_t$ is differentiation:
\begin{enumerate}
  \item for any smooth path $t\mapsto g_t\in G=\AdS$ with $g_0=1$,
  $$r_t \cdot I(g_t) \,\xrightarrow[t\to 0]{}\, i\bigg(\frac{\D}{\D t}\Big|_{t=0}\, g_t\bigg) \in \PP^3(\RR) ;$$
  \item for any smooth path $t\mapsto (h_t,k_t)\in G\times G=\mathrm{Isom}(\AdS)_0$ with~$h_0=\nolinebreak k_0$,
 $$r_t\,I_{\ast}(h_t,k_t)\,r_t^{-1} \,\xrightarrow[t\to 0]{}\, i_{\ast}\bigg(h_0,\frac{\D}{\D t}\Big|_{t=0}\, h_t k_t^{-1}\bigg) \in \PGL_4(\RR).$$
\end{enumerate}

%%%%%%%%%%%%%%%%%%%%%%%%%%%%%%%%%%%
\subsection{Constructing crooked fundamental domains}

Let $\Gamma$ be a free group, $j\in\Hom(\Gamma,G_0)$ an injective and discrete representation, and $u :\nolinebreak \Gamma\to\nolinebreak\g$ a $j$-cocycle (\ie $u(\gamma_1\gamma_2)=u(\gamma_1)+\Ad (j(\gamma_1))u(\gamma_2)$ for all $\gamma_1,\gamma_2\in\Gamma$).
Suppose the group
$$\Gamma^{j,u} := \{ (j(\gamma),u(\gamma))~|~\gamma\in\Gamma\} $$
acts properly discontinuously on~$\RR^{2,1}$.
Then $M=\Gamma^{j,u}\backslash\RR^{2,1}$ is a Margulis spacetime; conversely, any Margulis spacetime is of this form by \cite{fg83}.
Suppose $j$ is convex cocompact.
By \cite{glm09,dgk13}, up to replacing $u$ with $-u$, the cocycle $u$ is \emph{uniformly contracting}, in the sense that
$$\sup_{\gamma\in\Gamma,\ \lambda(j(\gamma))>0 }\ \frac{\mathrm{d}}{\mathrm{d}t}\Big|_{t=0}\ \frac{\lambda(e^{tu(\gamma)}j(\gamma))}{\lambda(j(\gamma))} < 0,$$
where $\lambda : G\rightarrow\RR_+$ is the translation length function of \eqref{eqn:def-lambda}.
By \cite[Th.\,1.5]{dgk-strips}, the cocycle $u$ is a negative \emph{infinitesimal strip deformation} of~$j$, which implies \cite[\S\,7.4]{dgk-strips} the existence of
\begin{itemize}
  \item a $j(\Gamma)$-invariant collection $\A$ of geodesic lines of~$\HH^2$ with pairwise disjoint closures in $\HH^2\cup\partial_{\infty}\HH^2$, dividing $\HH^2$ into fundamental domains for $j(\Gamma)$,
  \item a $(j,u)$-equivariant vector field $v : \bigcup_{\alpha\in\A}\alpha\rightarrow T(\HH^2)$ whose restriction to any $\alpha\in\A$ is a Killing vector field $v_{\alpha}\in\g$ with the following property: for any adjacent $\alpha\neq\alpha'$ in~$\A$,
  \begin{equation}\label{eqn:Mink-SQ}
  v_{\alpha'} - v_{\alpha} \in -\mathcal{SQ}(\alpha')+\mathcal{SQ}(\alpha),
  \end{equation}
  where $\mathcal{SQ}(\alpha),\mathcal{SQ}(\alpha')$ are the respective \emph{stem quadrants} (see Remark\linebreak \ref{rem:Einstein}.(1)) of $\alpha,\alpha'$ for their transverse orientations away from each~other.
\end{itemize}
Note that \eqref{eqn:Mink-SQ} implies the existence of a constant $k>0$ such that for any $x\in\alpha$ and $x'\in\alpha'$,
\begin{equation}\label{eqn:v-brings-closer}
\frac{\mathrm{d}}{\mathrm{d}t}\Big|_{t=0} d\big(\exp_x(tv(x), \exp_{x'}(tv(x')\big) \,\leq\, - k
\end{equation}
(infinitesimal analogue of \eqref{eqn:f-brings-closer}).
Indeed, for any oriented geodesic line $\ell$ of~$\HH^2$ and any Killing field $X$ of~$\HH^2$, the signed projection of $X(p)$ to~$\ell$ does not depend on $p\in\ell$; we call it the \emph{component of $X$ along~$\ell$}.
With this terminology, the left-hand side of \eqref{eqn:v-brings-closer} is equal to the component of $v_{\alpha'}-v_{\alpha}$ along the geodesic line through $x$ and~$x'$, oriented from $x$ to~$x'$ \cite[Rem.\,3.4]{dgk13}.
This component is positive by \eqref{eqn:Mink-SQ}, and in fact bounded from below by some positive constant independent of $x,x'$ since $\alpha$ and~$\alpha'$ have disjoint closures in $\HH^2\cup\partial_{\infty}\HH^2$ (see the proof of \cite[Lem.\,7.3]{dgk-strips}).
Thus \eqref{eqn:v-brings-closer} holds for some $k>0$.
Moreover, $k$ can taken uniform for all pairs of adjacent $\alpha\neq\alpha'$ in~$\A$, since $\A$ is finite modulo $j(\Gamma)$.

For any $\alpha\in\A$, let $\CP^*(\alpha)\subset\RR^{2,1}\simeq\g$ be the right crooked plane centered at $0\in\g$ associated with~$\alpha$: this is the set of Killing vector fields of~$\HH^2$ with a nonattracting fixed point in the closure $\overline{\alpha}$ of~$\alpha$ in $\HH^2\cup\partial_{\infty}\HH^2$.
Note that the exponential of $\CP^*(\alpha)\subset\g$ is the right $\AdSS$ crooked plane $\ACP^*(\alpha)\in G_0$ of Section~\ref{subsec:ACP}; thus
$$r_t \cdot I(\ACP^*(\ell)) \,\xrightarrow[t\to 0]{}\, i(\CP^*(\ell))$$
in the topology of Hausdorff convergence on compact subsets (see Section~\ref{subsec:embedP3}).
By a result of Drumm \cite{dru92} (see \cite[\S\,7.4]{dgk-strips}), the inclusions \eqref{eqn:Mink-SQ} imply that the right Minkowski crooked planes $\CP^*(\alpha)+v_{\alpha}$, for $\alpha\in\A$, are pairwise disjoint and divide $\RR^{2,1}$ into fundamental domains for~$\Gamma^{j,u}$.
(Compare with Theorem~\ref{thm:charact-crooked-domains}.)

Now let $(j_t)_{t\geq 0}\subset\Hom(\Gamma,G_0)$ be any smooth path of representations of~$\Gamma$ with $j_0=j$ and $\frac{\mathrm{d}}{\mathrm{d}t}|_{t=0}\ j_t(\gamma)\,j(\gamma)^{-1}=u(\gamma)$ for all $\gamma\in\Gamma$. 
Then the action of $\Gamma^{j,j_t}$ on $\AdS$ converges to that of $\Gamma^{j,u}$ on $\RR^{2,1}$ in the sense of Section~\ref{subsec:embedP3}: for any $\gamma\in\Gamma$,
$$r_t\,I_{\ast}\big(j(\gamma),j_t(\gamma)\big)\,r_t^{-1} \ \xrightarrow[t\to 0]{}\ i_{\ast}\big(j(\gamma), u(\gamma)\big) \in \PGL_4(\RR).$$
\noindent We use Theorem~\ref{thm:charact-crooked-domains} to construct fundamental domains for the $\Gamma^{j,j_t}$ actions in $\AdS$.
Choose a system $\A'\subset\A$ of representatives of $\A$ modulo $j(\Gamma)$, and let $f_t : \bigcup_{\alpha\in\A} \alpha\rightarrow\HH^2$ be the $(j,j_t)$-equivariant map whose restriction to any $\alpha'\in\A'$ coincides with the orientation-preserving isometry $\exp(tv_{\alpha'})\in G_0$ of~$\HH^2$.
Then $f_t$ converges to the identity of $\bigcup_{\alpha\in\A} \alpha$ as $t\rightarrow 0$, uniformly on compact sets, and $\frac{\mathrm{d}}{\mathrm{d}t}\big|_{t=0}\ f_t=v$.
Moreover, from \eqref{eqn:v-brings-closer} we see that for $t>0$ small enough, we have
$$d\big(f_t(x),f_t(x')\big) \leq d(x,x') - \frac{k}{2}\,t$$
for all adjacent $\alpha\neq\alpha'$ in~$\A$, all $x\in\alpha$, and all $x'\in\alpha'$.
The inequality still holds for any (possibly nonadjacent) $\alpha\neq\alpha'$ in~$\A$ by Remark~\ref{rem:adjacent-edges}.
Hence by Theorem~\ref{thm:charact-crooked-domains}, for all small enough $t>0$, the group $\Gamma^{j,j_t}$ acts properly discontinuously on $\AdS$; if $g_{\alpha,t}\in G_0$ is the unique orientation-preserving isometry of~$\HH^2$ whose restriction to~$\alpha$ coincides with~$f_t$, then the right $\AdSS$ crooked planes $g_{\alpha,t}\ACP^*(\alpha)$, for $\alpha\in\A$, are pairwise disjoint and divide $\AdS$ into fundamental domains for~$\Gamma^{j,j_t}$.
By construction, using Section~\ref{subsec:embedP3}, the following holds.

\begin{proposition}
As $t\rightarrow 0$, the $I_*(\Gamma^{j,j_t})$-invariant collection of ``rescaled" $\AdSS$ crooked planes $\{r_t\cdot I(g_{\alpha,t}\ACP^*(\alpha))\}_{\alpha\in\A}$ converges in $\PP^3(\RR)$ to the $i_*(\Gamma^{j,u})$-invariant collection of Minkowski crooked planes $\{i(\CP^*(\alpha)+v_{\alpha})\}_{\alpha\in\A}$.
\end{proposition}

In particular, any choice of finitely many arcs $\alpha_1, \ldots, \alpha_m \in \A$ bounding a fundamental domain for the action of $j(\Gamma)$ on $\HH^2$ determines a path of collapsing fundamental domains in $\AdS$ for the actions $\Gamma^{j,j_t}$, bounded by the $\AdSS$ crooked planes associated to those arcs.
The corresponding fundamental domains in $r_t\cdot I(\AdS)$ for the actions $I_*(\Gamma^{j,j_t})$ converge in projective space to a fundamental domain, bounded by the crooked planes $i(\CP^*(\alpha_1)+v_{\alpha_1}),\ldots, i(\CP^*(\alpha_m)+v_{\alpha_m})$, for the action of $i_*(\Gamma^{j,u})$ on $i(\RR^{2,1})$.

%%%%%%%%%%%%%%%%%%%%%%%%%%%%%%%%%%%%%%%%%%%%%%%%%%%
\section{Disjointness of $\AdSS$ crooked planes via stem~quadrants}\label{app:SQ}

With the notation of Remark~\ref{rem:Einstein}.(1), Drumm--Goldman's original disjointness criterion for crooked planes in~$\RR^{2,1}$ \cite{dg99} states that for any geodesic lines $\ell,\ell'$ of~$\HH^2$ and any $v,v'\in\RR^{2,1}$, the left crooked planes $\CP(\ell)+v$ and $\CP(\ell')+v'$ are disjoint in~$\RR^{2,1}$ if and only if $\ell$ and~$\ell'$ do not meet in~$\HH^2$ and $v'-v\in\mathcal{SQ}(\ell')-\mathcal{SQ}(\ell)$ for the transverse orientations of $\ell$ and~$\ell'$ away from each other (see \cite[Th.\,6.2]{bcdg13}).
To push the analogy between $\RR^{2,1}$ and $\AdS$ further, we now establish an analogue of this criterion for $\AdSS$ crooked planes, in terms of the $\AdSS$ stem quadrants $\SQ(\ell),\SQ(\ell')$ of Section~\ref{subsec:ACP}.
This is~a~com\-plement to Theorems \ref{thm:AdS-disj-left} and~\ref{thm:AdS-disj-right}; it is not needed anywhere in the paper.

\begin{theorem}\label{thm:AdS-disj-SQ-left}
Let $\ell,\ell'$ be geodesic lines of~$\HH^2$ and let $g,g'\in G_0$.
Consider the following conditions:
\begin{enumerate}[(A)]
  \item \label{ap3} the left $\AdSS$ crooked planes $g\ACP(\ell)$ and $g'\ACP(\ell')$ are disjoint,
  \item \label{ap4} the geodesic lines $g\cdot\ell$ and $g'\cdot\ell'$ are disjoint, and $g^{-1}g'\in \SQ(\ell)^{-1} \SQ(\ell')$ for the transverse orientations of $\ell$ and~$\ell'$ taken by $g$ and~$g'$ to transverse orientations of $g\cdot\ell$ and $g'\cdot\ell'$ away from each other.
\end{enumerate}
Then $\eqref{ap3}\Rightarrow\eqref{ap4}$ always holds, and $\eqref{ap4}\Rightarrow\eqref{ap3}$ holds under the assumption that $\ell$ and~$\ell'$ are disjoint and transversely oriented away from each other.
\end{theorem}

By Remark~\ref{rem:conj-ACP}, the left $\AdSS$ crooked plane $g\ACP(\ell)$ is the inverse of the right $\AdSS$ crooked plane $g^{-1}\,\ACP^*(g\cdot\ell)$, and similarly for $g'\ACP(\ell')$ and ${g'}^{-1}\ACP^*(g'\cdot\ell')$.
Moreover, $\SQ(g\cdot\ell)=g\SQ(\ell)g^{-1}$ and $\SQ(g'\cdot\ell')=g'\SQ(\ell'){g'}^{-1}$ for the transverse orientations of $g\cdot\ell$ and $g'\cdot\ell'$ induced by that of $\ell$ and~$\ell'$ via $g$ and~$g'$.
Therefore, Theorem~\ref{thm:AdS-disj-SQ-right} is equivalent to the following.

\begin{theorem}\label{thm:AdS-disj-SQ-right}
Let $\ell,\ell'$ be geodesic lines of~$\HH^2$ and let $g,g'\in G_0$.
Consider the following conditions:
\begin{enumerate}[(A)] \setcounter{enumi}{2}
  \item \label{ap1} the right $\AdSS$ crooked planes $g\ACP^*(\ell)$ and $g'\ACP^*(\ell')$ are disjoint,
  \item \label{ap2} the geodesic lines $\ell$ and~$\ell'$ are disjoint, and $g^{-1}g'\in \SQ(\ell) \SQ(\ell')^{-1}$ for the transverse orientations of $\ell$ and~$\ell'$ away from each other.
\end{enumerate}
Then $\eqref{ap1}\Rightarrow\eqref{ap2}$ always holds, and $\eqref{ap2}\Rightarrow\eqref{ap1}$ holds under the assumption that $g$ and~$g'$ take $\ell$ and~$\ell'$ to disjoint geodesics transversely oriented away from each other. 
\end{theorem}

Recall from Theorem~\ref{thm:AdS-disj-right} that the right $\AdSS$ crooked planes $g\ACP^*(\ell)$ and $g'\ACP^*(\ell')$ are disjoint (\ie \eqref{ap1} holds) if and only if
\begin{equation}
 \sup_{(x,x')\in\ell\times\ell'} d(g\cdot x, g'\cdot x') - d(x,x') < 0 .
\label{eqn:re-C}
 \end{equation}
Thus $\eqref{ap2}\Rightarrow\eqref{ap1}$ is not difficult to prove using hyperbolic trigonometry, under the assumption that $g$ and~$g'$ take $\ell$ and~$\ell'$ to disjoint geodesics transversely oriented away from each other: see \cite[Prop.\,8.2]{dgk-strips}.
Note that $\eqref{ap2}\Rightarrow\eqref{ap1}$ fails if we remove this assumption: indeed, to see that $g\cdot\ell\cap g'\cdot\ell'=\emptyset$ is required, 
fix $(x,x')\in\ell\times \ell'$ and
consider a long translation $g\in\nolinebreak\SQ(\ell)^{-1}$ (where $\ell$ is transversely oriented away from~$\ell'$) whose attracting fixed point is an endpoint of~$\ell'$. If $g$ is long enough, then for all $g'\in\SQ(\ell')^{-1}$ very close to $e\in G_0$ the lines $g\cdot \ell$ and $g'\cdot \ell'$ intersect, and $d(g\cdot x, g'\cdot x')-d(x,x')>0$, hence \eqref{eqn:re-C} fails.
Similarly, if the translation axis of $g\in\SQ(\ell)^{-1}$ intersects~$\ell'$, then \eqref{eqn:re-C} still fails for $g$ long enough and for $g'\in\SQ(\ell')^{-1}$ close to $e\in G_0$; the line $g\cdot \ell$ is then disjoint from $g'\cdot \ell'$ but its transverse orientation points towards $g'\cdot \ell'$.

\medskip

We shall now prove $\eqref{ap1}\Rightarrow\eqref{ap2}$ by direct computations.
Before engaging in those, it may be encouraging to note that if we are in the limit case of~\eqref{ap2} then we are also in the limit case of~\eqref{ap1}: namely, if $(g,g')\in \partial\SQ(\ell) \times \partial\SQ(\ell')$, then $g$ (\resp $g'$) is a parabolic element fixing an endpoint $\xi$ of~$\ell$ (\resp $\xi'$ of~$\ell'$), hence $d(g\cdot\xi, g'\cdot\xi') - d(\xi,\xi') = 0$ (see Observation~\ref{obs:ext-diff-d} and Remark~\ref{rem:concrete-Fgg'} for the meaning of the left-hand side).
Here, and in the whole proof, we denote by $\partial\SQ(\ell)\subset\AdS$ the boundary of $\SQ(\ell)$ in the timelike plane of $\AdS$ that contains it, and similarly for $\partial\SQ(\ell')$.

%%%%%%%%%%%%%%%%%%%%%%%%%%%%%%%%%%
\subsection{Proof of $\eqref{ap1}\Rightarrow\eqref{ap2}$ in Theorem~\ref{thm:AdS-disj-SQ-right}}\label{subsec:ap1}

If \eqref{ap1} holds, then the geodesic lines $\ell$ and~$\ell'$ have disjoint closures in $\HH^2\cup\partial_{\infty}\HH^2$ by Theorem~\ref{thm:AdS-disj-right}.
Thus, working in the upper half-plane model of~$\HH^2$, we may assume that $\ell$ has endpoints $(\xi_+, \xi_-)=(-R^{-1}, R^{-1})$ and $\ell'$ has endpoints $(\xi'_+,\xi'_-)=(R, -R)$, for some $R>1$.
We endow $\ell$ and~$\ell'$ with transverse orientations away from each other, pointing to $0$ and~$\infty$ respectively.
We write elements of $G_0=\PSL_2(\RR)\subset\PGL_2(\RR)=G$ as square matrices, whose entries are defined up to a global multiplicative factor.

%%%%%%%%%%%%%%%%%%%%%%%%%%%%

\begin{claim}\label{cla:Nis}
An element $h\in G_0$ belongs to $\SQ(\ell)$ (\resp $\SQ(\ell')^{-1}$) if and only if it is of the form
$$h=\begin{bmatrix} \alpha & -vR^{-1} \\ vR & \beta \end{bmatrix} \quad\quad \left( \text{resp. } h=\begin{bmatrix} \alpha & vR \\ -vR^{-1} & \beta \end{bmatrix} \right)$$
where $\alpha,\beta,v\in\RR$ satisfy
$$\left\{\begin{array}{l}
|\alpha| < |\beta|,\\
2|v| < |\alpha-\beta|.
\end{array} \right .$$
\end{claim}

\begin{proof}
It is enough to prove the claim for $\SQ(\ell)$: the characterization of $\SQ(\ell')^{-1}$ is similar, replacing $R$ with~$R^{-1}$.
The elements of $\SQ(\ell)$ are the hyperbolic elements of~$G_0$ whose translation axis is orthogonal to~$\ell$, oriented downwards (away from $\infty$). 
The condition $2|v| < |\alpha-\beta|$ expresses the fact that $\mathrm{tr}(h)^2>4\, \mathrm{det}(h)$, \ie $h$ is hyperbolic.
The condition that the ratio of the nondiagonal entries of~$h$ be $-R^{-2}$ expresses the fact that the fixed points of the M\"obius transformation $h$ have product $R^{-2}$, \ie are symmetric with respect to the line~$\ell$: this means $h\in\SQ(\ell)\cup \SQ(\ell)^{-1}$.
Finally, the condition $|\alpha|<|\beta|$ expresses the fact that the attracting fixed point of $h$ lies below~$\ell$, \ie $h\in \SQ(\ell)$.
Indeed, to see this, we argue by continuity and by connectedness of $\SQ(\ell)$.
Start from the case $v=0$, where $h$ is diagonal: in that case, clearly $h\in\SQ(\ell)$ if and only if  $|\alpha|<|\beta|$.
As we deform~$h$, reaching $\alpha=\beta$ would contradict $2|v|<|\alpha-\beta|$, and reaching $\alpha=-\beta$ would contradict the conjunction of $2|v|<|\alpha-\beta|$ with the determinant condition $\alpha\beta+v^2>0$.
Thus $|\alpha|<|\beta|$ always holds on $\SQ(\ell)$ (and $|\alpha|>|\beta|$ always holds on $\SQ(\ell)^{-1}$).
\end{proof}

%%%%%%%%%%%%%%%%%%%%%%%%%%%%

Recall from Theorem~\ref{thm:AdS-disj-right} that Condition~\eqref{ap1} of Theorem~\ref{thm:AdS-disj-SQ-right} is equivalent to the fact that
\begin{equation}\label{eqn:bring-closer}
d(g\cdot \xi_{\varepsilon}, g'\cdot\xi'_{\varepsilon'}) - d(\xi_{\varepsilon},\xi'_{\varepsilon'}) < 0
\end{equation}
for all $(\varepsilon,\varepsilon')\in\{+,-\}^2$.
Let us write
$$h := g^{-1}g' = \begin{bmatrix} a & b \\ c & d \end{bmatrix} \in \PGL_2(\RR),$$
where $a,b,c,d\in\RR$ satisfy $ad-bc=1$.

\begin{claim}\label{cla:bring-closer}
For any $(\varepsilon,\varepsilon')\in\{+,-\}^2$, the inequality \eqref{eqn:bring-closer} is equivalent to
\begin{equation}\label{eqn:explicit-bring-closer}
| aR+\varepsilon' b + \varepsilon c + \varepsilon \varepsilon ' dR^{-1} | < | R+\varepsilon\varepsilon' R^{-1} |.
\end{equation}
\end{claim}

\begin{proof}
For $t\in\RR$, we set
$$D := \begin{bmatrix} R^{1/2} & 0 \\ 0 & R^{-1/2} \end{bmatrix} \quad\text{ and }\quad L_t := \begin{bmatrix} \cosh (t/2) & \sinh (t/2) \\ \sinh(t/2) & \cosh(t/2) \end{bmatrix},$$
so that $p_t:= D^{-1}L_t^{-1}\cdot \sqrt{-1} \underset{t\rightarrow \varepsilon\infty}\longrightarrow  \xi_\varepsilon$ and $p'_{t'}:= DL_{t'}\cdot \sqrt{-1} \underset{t'\rightarrow \varepsilon'\infty}{\longrightarrow}  \xi'_{\varepsilon'}$ define unit-speed parameterizations of $\ell$ and $\ell'$ respectively.
Using the identity
$$2\cosh d(\sqrt{-1}, u\cdot \sqrt{-1}) = \Vert u \Vert^2 := \frac{\alpha^2+\beta^2+\gamma^2+\delta^2}{(\alpha\delta - \beta\gamma)^2}$$
for all $u=\begin{bmatrix} \alpha & \beta \\ \gamma & \delta \end{bmatrix} \in G_0$, we find that for any $t,t'\in\RR$,
\begin{eqnarray*}
& 2\cosh d(g\cdot p_t, g'\cdot p'_{t'}) = \Vert L_{t} D h D L_{t'} \Vert^2\\
= & (R^2 a^2 + b^2 + c^2 +R^{-2} d^2) \cosh t \cosh t' + 2 (R ab + R^{-1}cd) \cosh t \sinh t'\\ 
& \ +\ 2(R ac + R^{-1}bd) \sinh t \cosh t' + 2 (ad + bc) \sinh t \sinh t' .
\end{eqnarray*}
Similarly, replacing $h$ with the identity matrix, we find
$$ 2\cosh d(p_t, \cdot p'_{t'}) = (R^2 +R^{-2} ) \cosh t \cosh t' + 2 \sinh t \sinh t' .$$
Taking asymptotics as $t\rightarrow \varepsilon\infty$ and $t'\rightarrow \varepsilon'\infty$ and using the continuity in Observation~\ref{obs:ext-diff-d}, we finally obtain 
\begin{eqnarray*}
d(g\cdot \xi_\varepsilon, g'\cdot \xi'_{\varepsilon'}) - d(\xi_\varepsilon, \xi'_{\varepsilon'}) &=&
\lim_{\substack{t\to \varepsilon\infty \\ t' \to \varepsilon'\infty}} \log \frac{2\cosh d(g\cdot p_t,g'\cdot p'_{t'})}{2\cosh d(p_t, p'_{t'})} \\ &=& \log \frac{(aR+\varepsilon' b+\varepsilon c+\varepsilon \varepsilon' d R^{-1})^2}{(R + \varepsilon \varepsilon' R^{-1})^2},
\end{eqnarray*}
from which the claim immediately follows.
\end{proof}

Our goal now is to prove that if $h$ satisfies \eqref{eqn:explicit-bring-closer} for all $(\varepsilon,\varepsilon')\in\{+,-\}^2$, then $h\in\SQ(\ell)\SQ(\ell')^{-1}$.
We begin by proving something very close, namely that (in the generic case) $h$ belongs to $\partial\SQ(\ell)\,\SQ(\ell')^{-1}$ or $\SQ(\ell)\,\partial\SQ(\ell')^{-1}$.
We shall then pass to $\SQ(\ell)\SQ(\ell')^{-1}$ by a small perturbation argument.

For $\epsilon\in\{ +,-\}$, we define one-parameter subgroups $P_{\epsilon}=(P_{\epsilon}(t))_{t\in\RR}$ and $P'_{\epsilon}=(P'_{\epsilon}(t))_{t\in\RR}$ of~$G_0$ as follows:
$$P_{\epsilon}(t) := \begin{bmatrix} 1-t & -\epsilon R^{-1}t \\ \epsilon Rt & 1+t\end{bmatrix} \quad\mathrm{and}\quad P'_{\epsilon}(t) := \begin{bmatrix} 1+t & -\epsilon Rt \\ \epsilon R^{-1}t & 1-t\end{bmatrix}.$$
The group $P_{\epsilon}$ (\resp $P'_{\epsilon}$) consists of parabolic elements fixing $\xi_{\epsilon}$ (\resp $\xi'_{\epsilon}$) and we have
\begin{equation}
\label{eqn:parabound}
\left\{ \begin{array}{ccl}
\partial\SQ(\ell) & = & (P_+(t))_{t\geq 0} \cup (P_-(t))_{t\geq 0},\\
\partial\SQ(\ell') & = & (P'_+(t))_{t\geq 0} \cup (P'_-(t))_{t\geq 0}.
\end{array}\right.
\end{equation}

\begin{claim}\label{cla:almostSQ}
Suppose that $h$ satisfies \eqref{eqn:explicit-bring-closer} for all $(\varepsilon,\varepsilon')\in\{+,-\}^2$ and that it is not a diagonal matrix. 
Then one of the rays $(hP'_+(t))_{t>0}$ or $(hP'_-(t))_{t>0}$ intersects $\SQ(\ell)$ transversely, or one of the rays $(P_+(t)^{-1}h)_{t>0}$ or $(P_-(t)^{-1}h)_{t>0}$ intersects $\SQ(\ell')^{-1}$ transversely.
\end{claim}

\begin{proof}
We see $\AdS=G_0=\PSL_2(\RR)$ as an open domain in the set $\PP(\mathrm{M}_2(\RR))$ of projectivized nonzero $2\times 2$ real matrices.
For any subset $X$ of $\PP(\mathrm{M}_2(\RR))$, we denote by $\widehat{X}\subset\PP(\mathrm{M}_2(\RR))$ its projective span:
\begin{itemize}
  \item $\widehat{\SQ(\ell)}=\widehat{\SQ(\ell)^{-1}}$ (\resp $\widehat{\SQ(\ell')}=\widehat{\SQ(\ell')^{-1}}$) is the projective plane characterized by a ratio of nondiagonal entries equal to $-R^2$ (\resp $-R^{-2}$), as in Claim~\ref{cla:Nis};
  \item $\widehat{P_+}$ is a projective line, equal to the union of the group $P_+$ and of a singleton $\{P_+(\infty)\}$ in the boundary of $\AdS$ in $\PP(\mathrm{M}_2(\RR))$; and similarly for $\widehat{P_-},\widehat{P'_+},\widehat{P'_-}$.
\end{itemize}

We first suppose that $h\notin\widehat{\SQ(\ell)}\cup\widehat{\SQ(\ell')}$, so that $bR+cR^{-1}\neq 0$ and $cR+bR^{-1}\neq 0$.
For $\epsilon\in\{ +,-\}^2$, the projective line $h\widehat{P'_{\epsilon}}$ then intersects the projective plane $\widehat{\SQ(\ell)}$ transversely at a unique point $hP'_{\epsilon}(t'_{\epsilon})$, and similarly $\widehat{P_{\epsilon}} h$ intersects $\widehat{\SQ(\ell')^{-1}}$ transversely at a unique point $P_{\epsilon}(t_{\epsilon})^{-1}h$, where $t_{\epsilon},t'_{\epsilon}\in\PP^1(\RR)\smallsetminus\{ 0\}$ can both be found by solving a degree-one equation: namely,
\begin{equation}\label{eqn:tis}
\begin{array}{ccl}
t'_{\epsilon} & = & \displaystyle{\frac{bR+cR^{-1}}{R(\epsilon aR+b)-R^{-1}(c+\epsilon dR^{-1})}},\\
t_{\epsilon} & = & \displaystyle{\frac{cR+bR^{-1}}{R(\epsilon aR+c)-R^{-1}(b+\epsilon dR^{-1})}}.
\end{array}
\end{equation}
Using $ad-bc=1$, we find
$$h P'_{\epsilon}(t'_{\epsilon}) =\! \begin{bmatrix} \alpha & -vR^{-1} \\ vR & \beta \end{bmatrix} \!\in \widehat{\SQ(\ell)},
\quad
P_{\epsilon}(t_{\epsilon})^{-1} h =\! \begin{bmatrix} \alpha' & v'R \\ -v'R^{-1} & \beta' \end{bmatrix} \!\in \widehat{\SQ(\ell')^{-1}},$$
where
$$\left \{ \begin{array}{l} 
\alpha = (\epsilon aR+b)^2 - R^{-2},\\ 
\beta = R^2 - (c+\epsilon dR^{-1})^2,\\ 
v = acR + 2\epsilon bc + bdR^{-1},
\end{array} \right .
\quad\quad
\left \{ \begin{array}{l} 
\alpha' = (\epsilon aR+c)^2 - R^{-2},\\ 
\beta' = R^2 - (b+\epsilon dR^{-1})^2,\\ 
v' = abR + 2\epsilon bc + cdR^{-1}.
\end{array} \right .$$
Note that if $t'_{\epsilon}\neq\infty$, then $hP'_{\epsilon}(t'_{\epsilon})\in\widehat{\SQ(\ell)}$ actually belongs to $\SQ(\ell)$: indeed, by Claim~\ref{cla:Nis}, this amounts to checking that $|\alpha|<|\beta|$ and $2|v|<|\alpha-\beta|$.
We use $ad-bc=1$ again to compute
$$\left\{ \begin{array}{l}
\alpha+\beta ~=~ R^2-R^{-2}+(aR+\epsilon b+\epsilon c+dR^{-1})(aR+\epsilon b-\epsilon c-dR^{-1}),\\
\beta-\alpha+2\epsilon v ~=~ (R-R^{-1})^2 - (aR+\epsilon b-\epsilon c-dR^{-1})^2,\\
\beta-\alpha-2\epsilon v ~=~ (R+R^{-1})^2 - (aR+\epsilon b+\epsilon c+dR^{-1})^2.
\end{array}\right.$$
These three real numbers are all positive by \eqref{eqn:explicit-bring-closer}, hence $hP'_{\epsilon}(t'_{\epsilon})\in\SQ(\ell)$.
Similarly, if $t_{\epsilon}\neq\infty$, then $P_{\epsilon}(t_{\epsilon})^{-1}h\in\widehat{\SQ(\ell')^{-1}}$ actually belongs to $\SQ(\ell')^{-1}$ (just exchange $b$ and~$c$).
Therefore, the claim reduces to proving that at least one of $t_+, t_-, t'_+, t'_-\in\PP^1(\RR)$ is a positive real.
Since they are all nonzero, they have well-defined inverses in~$\RR$.
The number
$$\frac{1}{t'_+} + \frac{1}{t'_-} = 2\frac{bR-cR^{-1}}{bR+cR^{-1}}$$
is positive if and only if $c=0$ or $|b/c|>R^{-2}$; the number
$$\frac{1}{t_+} + \frac{1}{t_-} = 2\frac{cR-bR^{-1}}{cR+bR^{-1}}$$
is positive if and only if $c\neq 0$ and $|b/c|<R^2$.
Since $R>1$, at least one of these two numbers must be positive, hence at least one of $t_+, t_-, t'_+, t'_-$ is a positive real.

Suppose now that $h\in\widehat{\SQ(\ell')}\smallsetminus\widehat{\SQ(\ell)}$.
For $\epsilon\in\{ +,-\}$, the projective line $h\widehat{P'_{\epsilon}}$ still intersects the projective plane $\widehat{\SQ(\ell)}$ transversely at a unique point $hP'_{\epsilon}(t'_{\epsilon})$, where $t'_{\epsilon}\in\PP^1(\RR)\smallsetminus\{ 0\}$ is given by \eqref{eqn:tis}.
Since $h\in\widehat{\SQ(\ell')}$, we have $b/c=-R^2$ by Claim~\ref{cla:Nis}; in particular, $|b/c|>\nolinebreak R^{-2}$ since $R>1$, and so one of $t'_+$ or~$t'_-$ is a positive real by the above calculation.
The case $h\in\widehat{\SQ(\ell)}\smallsetminus\widehat{\SQ(\ell')}$ is similar.
The case $h\in\widehat{\SQ(\ell)}\cap\widehat{\SQ(\ell')}$ is ruled out by assumption, since it corresponds to a diagonal matrix (see Claim~\ref{cla:Nis}).
\end{proof}
 
Suppose that Condition~\eqref{ap1} of Theorem~\ref{thm:AdS-disj-SQ-right} holds, \ie $h=g^{-1}g'\in G_0$ satisfies \eqref{eqn:explicit-bring-closer} for all $(\varepsilon,\varepsilon')\in\{+,-\}^2$.
If $h$ is a diagonal matrix, then it is by Claim~\ref{cla:bring-closer} a downward translation (of length at most $4\log R$) along the axis $(\infty, 0)$, perpendicular to $\ell$ and~$\ell'$.
We can break up $h$ into two translations of length twice less, which belong to $\SQ(\ell)$ and $\SQ(\ell')^{-1}$, yielding \eqref{ap2}.

If $h$ is not a diagonal matrix, then we use Claim~\ref{cla:almostSQ}: suppose for instance that the ray $(hP'_+(t))_{t>0}$ intersects $\SQ(\ell)$ transversely.
If we perturb slightly the parabolic semigroup $(P'_+(t))_{t>0}$ to a hyperbolic semigroup of $\SQ(\ell')$ translating along a line close to~$\xi'_+$, then the transversality property remains true, showing that $h\in\SQ(\ell)\SQ(\ell')^{-1}$.
Similarly, if the ray $(hP'_-(t))_{t>0}$ intersects $\SQ(\ell)$ transversely, or if one of the rays $(P_+(t)^{-1}h)_{t>0}$ or $(P_-(t)^{-1}h)_{t>0}$ intersects $\SQ(\ell')^{-1}$ transversely, then $h\in\SQ(\ell)\SQ(\ell')^{-1}$, \ie \eqref{ap2} holds.
This completes the proof of Theorem~\ref{thm:AdS-disj-SQ-right} (hence also of the equivalent Theorem~\ref{thm:AdS-disj-SQ-left}).

%%%%%%%%%%%%%%%%%%%%%%%%%%%%%%%%%%%
\subsection{A consequence of the proof}

We now work in the setting of Theorem~\ref{thm:AdS-disj-SQ-left}, \ie with left instead of right $\AdSS$ crooked planes.
Here is an interesting consequence of the proof of the previous paragraph.
Let $\ell,\ell'$ be two geodesic lines of~$\HH^2$ with disjoint closures in $\HH^2\cup\partial_{\infty}\HH^2$, transversely oriented away from each other.
Let
$$\mathscr{C} := \left\{h\in G_0~\Big|~\inf_{(x,x')\in\ell\times \ell'} d(x,h\cdot x')-d(x,x')>0\right\}$$
be the set of elements $h$ for which the left $\AdSS$ crooked planes $\ACP(\ell)$ and $h\ACP(\ell')$ are disjoint (Theorem~\ref{thm:AdS-disj-left}).
If $h\in\mathscr{C}$, then $\ell$ and $h\cdot\ell'$ have disjoint closures in $\HH^2\cup\partial_{\infty}\HH^2$.
Therefore $\mathscr{C}$ has four connected components, according to the configuration of transverse orientations on $\ell$ and $h\cdot\ell'$.
The involution $h\mapsto \sigma_\ell h \sigma_{\ell'}$ (where $\sigma_L$ denotes the orthogonal reflection in a line~$L$) switches these components in pairs.
Let $\mathscr{C}_0$ be the component of elements $h$ such that the transverse orientations on $\ell$ and $h\cdot\ell'$ induced by those of $\ell$ and~$\ell'$ via~$h$ are still away from each other.

\begin{proposition}\label{prop:param}
Let $\ell,\ell'$ be two geodesic lines of~$\HH^2$ with disjoint closures in $\HH^2\cup\partial_{\infty}\HH^2$, transversely oriented away from each other.
Then the product map restricts to homeomorphisms
$$\SQ(\ell)^{-1} \times \partial\SQ(\ell') \overset{\scriptscriptstyle\sim}{\longrightarrow} \mathscr{C}_0 \quad\text{and}\quad \partial\SQ(\ell)^{-1} \times \SQ(\ell') \overset{\scriptscriptstyle\sim}{\longrightarrow} \mathscr{C}_0.$$
\end{proposition}

Recall from \eqref{eqn:parabound} that $\partial\SQ(\ell)\smallsetminus\{ e\}$ has two connected components, which we denote by $(P_+(t))_{t>0}$ and $(P_-(t))_{t>0}$; they are lightlike geodesic rays consisting of parabolic elements fixing an endpoint of~$\ell$.
Similarly, $\partial\SQ(\ell')\smallsetminus\nolinebreak\{ e\}$ has two connected components $(P'_+(t))_{t>0}$ and $(P'_-(t))_{t>0}$.
In order to prove Proposition~\ref{prop:param}, we first establish the following result, which refines Claim~\ref{cla:almostSQ} in the case that $\ell$ and $h\cdot\ell'$ have disjoint closures in $\HH^2\cup\partial_{\infty}\HH^2$ and are transversely oriented away from each other; it will be applied to inverses of elements of~$\mathscr{C}_0$.

\begin{claim}\label{cla:refine-almostSQ}
Let $\ell,\ell'$ be two geodesic lines of~$\HH^2$ and let $h\in G_0$ satisfy $\sup_{(x,x')\in\ell\times \ell'} d(x,h\cdot x')-d(x,x')<0$.
Suppose $\ell,\ell'$ have disjoint closures in $\HH^2\cup\partial_{\infty}\HH^2$ and are transversely oriented away from each other, and $\ell$ and $h\cdot\ell'$ also have disjoint closures in $\HH^2\cup\partial_{\infty}\HH^2$ and are transversely oriented away from each other, for the transverse orientation induced by~$h$.
Then either $h\in\SQ(\ell)$, or \emph{exactly one} of the open rays $(hP'_+(t))_{t>0}$ or $(hP'_-(t))_{t>0}$ intersects $\SQ(\ell)$ transversely.
Similarly, either $h\in\SQ(\ell')^{-1}$, or exactly one of the open rays $(P_+(t)^{-1}h)_{t>0}$ or $(P_-(t)^{-1}h)_{t>0}$ intersects $\SQ(\ell')^{-1}$ transversely.
\end{claim}

\begin{proof}[Proof of Claim~\ref{cla:refine-almostSQ}]
As in Section~\ref{subsec:ap1}, we work in the upper half-plane~mo\-del of~$\HH^2$, where we may assume that $\ell$ has endpoints $(\xi_+, \xi_-)=(-R^{-1}, R^{-1})$ and $\ell'$ has endpoints $(\xi'_+,\xi'_-)=(R, -R)$, for some $R>1$.
The transverse orientations of $\ell$ and~$\ell'$, away from each other, point to $0$ and~$\infty$ respectively.
Write
$$h = \begin{bmatrix} a & b \\ c & d \end{bmatrix} \in G_0\subset\PGL_2(\RR),$$
where $a,b,c,d\in\RR$ satisfy $ad-bc=1$.
Then $h$ satisfies \eqref{eqn:explicit-bring-closer} for all $(\varepsilon,\varepsilon')\in\{+,-\}^2$, by Theorem~\ref{thm:AdS-disj-right} and Claim~\ref{cla:bring-closer}.
Since $h\cdot\ell'$ lies on the negative side of~$\ell$, we have $|h\cdot R|>R^{-1}$ and $|h\cdot (-R)|>R^{-1}$, \ie
$$|aR+b| > |c+dR^{-1}| \text{ and } |aR-b| > |c-dR^{-1}|.$$
Since $h^{-1}\cdot\ell$ lies on the negative side of~$\ell'$, we have $|h^{-1}\cdot R^{-1}|<R$ and $|h^{-1}\cdot (-R^{-1})|<R$, \ie
$$|aR-c| > |b-dR^{-1}| \text{ and } |aR+c| > |b+dR^{-1}|.$$
Combining the four inequalities above and using the equivalence $|x|>\nolinebreak |y|\Leftrightarrow (x+y)(x-y)>0$, we see that the four real numbers
\begin{equation}\label{eqn:samesign}
\begin{array}{ll}
aR-b-c+dR^{-1}, & aR+b+c+dR^{-1},\\
aR+b-c-dR^{-1}, & aR-b+c-dR^{-1}
\end{array}
\end{equation}
all have the same sign.
In particular, since $R>1$, the denominators of $t'_+$ and~$t'_-$ in \eqref{eqn:tis} are nonzero and have opposite signs; since the numerators are the same, we see that either $t'_+>0>t'_-$, or $t'_+=0=t'_-$, or $t'_+<0<t'_-$.
The same holds for $t_+$ and~$t_-$.
Thus, by Claim~\ref{cla:almostSQ} and its proof, either $(hP'_+(t))_{t>0}$ intersects $\SQ(\ell)$ transversely (if $t'_+>0$), or $h\in\SQ(\ell)$ (if $t'_+=t'_-=0$), or $(hP'_-(t))_{t>0}$ intersects $\SQ(\ell)$ transversely (if $t'_->0$), and these three cases are mutually exclusive.
Similarly, either $(P_+(t)^{-1}h)_{t>0}$ intersects $\SQ(\ell')^{-1}$ transversely, or $h\in\SQ(\ell')^{-1}$, or $(P_-(t)^{-1}h)_{t>0}$ intersects $\SQ(\ell')^{-1}$ transversely, and these three cases mutually exclusive.
\end{proof}

\begin{proof}[Proof of Proposition~\ref{prop:param}]
Let $h\in\mathscr{C}_0$.
Applying Claim~\ref{cla:refine-almostSQ} to $(\ell,h\cdot\ell',h^{-1})$ instead of $(\ell,\ell',h)$, and using the fact that $\partial\SQ(h\cdot\ell')=h\,\partial\SQ(\ell')h^{-1}$ by compatibility of the transverse orientations, we see that either $h\in\SQ(\ell)^{-1}$ or there is exactly one connected component $(P'_t)_{t>0}$ of $\partial\SQ(\ell')\smallsetminus \{e\}$ such that $(h{P'_t}^{-1})_{t>0}$ intersects $\SQ(\ell)^{-1}$ transversely.
In particular, $h$ can be written \emph{uniquely} as the product of an element of $\SQ(\ell)^{-1}$ and of an element of $\partial\SQ(\ell')$.
By transversality, this decomposition depends continuously on~$h$, which shows that the product map $\SQ(\ell)^{-1}\times\partial\SQ(\ell')\rightarrow\mathscr{C}_0$ is a homeomorphism.
Similarly, the product map $\partial\SQ(\ell)^{-1}\times\SQ(\ell')\rightarrow\mathscr{C}_0$ is a homeomorphism.
\end{proof}

\vspace{0.5cm}
%%%%%%%%%%%%%%%%%%%%%%%%%%%%%%%%%%%%%%%%%%%%%%%%%%%


\begin{thebibliography}{GLMM}

\bibitem[A]{aus64}
\textsc{L. Auslander}, \textit{The structure of compact locally affine manifolds}, Topology~3 (1964), p.~131--139.

\bibitem[Ba]{bal95}
\textsc{W. Ballmann}, \textit{Lectures on spaces of nonpositive curvature},
%With an appendix by Misha Brin,
DMV Seminar~25, Birkh\"auser Verlag, Basel, 1995.

\bibitem[BCDG]{bcdg13}
\textsc{J.-P. Burelle, V. Charette, T. Drumm, W. M. Goldman}, \textit{Crooked half-spaces}, Enseign. Math.~(2), to appear.

\bibitem[Bu]{bus78}
\textsc{P. Buser}, \textit{The collar theorem and examples}, Manuscripta Math.~25 (1978), p.~349--357.

\bibitem[CDG1]{cdg10}
\textsc{V. Charette, T. Drumm, W. M. Goldman}, \textit{Affine deformations of the three-holed sphere},
Geom. Topol.~14 (2010), p.~1355--1382.

\bibitem[CDG2]{cdg11}
\textsc{V. Charette, T. Drumm, W. M. Goldman}, \textit{Finite-sided deformation spaces of complete affine $3$-manifolds}, J. Topol.~7 (2014), p.~225--246.

\bibitem[CDG3]{cdg13}
\textsc{V. Charette, T. Drumm, W. M. Goldman}, \textit{Proper affine deformations of two-generator Fuchsian groups}, in preparation.

\bibitem[ChaG]{cg00}
\textsc{V. Charette, W. M. Goldman}, \textit{Affine Schottky groups and crooked tilings}, in \textit{Crystallographic groups and their generalizations (Kortrijk, 1999)}, p.~69--98, Contemp. Math.~262, American Mathematical Society, Providence, RI, 2000.

\bibitem[CFL]{cfl13}
\textsc{V. Charette, D. Francoeur, R. Lareau-Dussault}, \textit{Fundamental polyhedra in the Einstein Universe}, preprint, arXiv:1307:6531.

\bibitem[ChoG]{cg13}
\textsc{S. Choi, W. M. Goldman}, \textit{Topological tameness of Margulis spacetimes}, preprint, 2013.

\bibitem[DGK1]{dgk13}
\textsc{J. Danciger, F. Gu\'eritaud, F. Kassel}, \textit{Geometry and topology of complete Lorentz spacetimes of constant curvature}, preprint, arXiv:1306.2240.

\bibitem[DGK2]{dgk-strips}
\textsc{J. Danciger, F. Gu\'eritaud, F. Kassel}, \textit{Margulis spacetimes via the arc complex}, preprint, arXiv:1407.5422.

\bibitem[DGK3]{dgk-parab}
\textsc{J. Danciger, F. Gu\'eritaud, F. Kassel}, \textit{Geometry, topology, and moduli spaces of Margulis spacetimes}, in preparation.

\bibitem[D]{dru92}
\textsc{T. Drumm}, \textit{Fundamental polyhedra for Margulis space-times}, Topology~21 (1992), p.~677--683.

\bibitem[DG1]{dg95}
\textsc{T. Drumm, W. M. Goldman}, \textit{Crooked Planes}, Electronic Research Announcements of the American Math. Soc.~1~(1) (1995).

\bibitem[DG2]{dg99}
\textsc{T. Drumm, W. M. Goldman}, \textit{The geometry of crooked planes}, Topology~38 (1999), p.~323--352.

\bibitem[F]{fra03}
\textsc{C. Frances}, \textit{The conformal boundary of Margulis space-times}, C.\ R.\ Math.\ Acad.\ Sci.\ Paris~336 (2003), p.~751--756.

\bibitem[FG]{fg83}
\textsc{D. Fried, W. M. Goldman}, \textit{Three-dimensional affine crystallographic groups}, Adv. Math.~47 (1983), p.~1--49.

\bibitem[G]{gol13}
\textsc{W. M. Goldman}, \textit{Crooked surfaces and anti-de Sitter geometry}, preprint, arXiv:1302.4911.

\bibitem[GLM]{glm09}
\textsc{W. M. Goldman, F. Labourie, G. A. Margulis}, \textit{Proper affine actions and geodesic flows of hyperbolic surfaces}, Ann. of Math.~170 (2009), p.~1051--1083.

\bibitem[GM]{gm00}
\textsc{W. M. Goldman, G. A. Margulis}, \textit{Flat Lorentz $3$-manifolds and cocompact Fuchsian groups}, in \textit{Crystallographic groups and their generalizations (Kortrijk, 1999)}, p.~135--145, Contemp. Math.~262, American Mathematical Society, Providence, RI, 2000.

\bibitem[GK]{gk13}
\textsc{F. Gu\'eritaud, F. Kassel}, \textit{Maximally stretched laminations on geometrically finite hyperbolic manifolds}, preprint, arXiv:1307.0256.

\bibitem[K]{kasPhD}
\textsc{F. Kassel}, \textit{Quotients compacts d'espaces homog\`enes r\'eels ou $p$-adiques}, PhD thesis, Universit\'e Paris-Sud~11, November~2009, available at \url{http://math.univ-lille1.fr/}$\sim$\url{kassel/}.

\bibitem[KR]{kr85}
\textsc{R. S. Kulkarni, F. Raymond}, \textit{3-dimensional Lorentz space-forms and Seifert fiber spaces}, J. Differential Geom.~21 (1985), p.~231--268.

\bibitem[Ma1]{mar83}
\textsc{G. A. Margulis}, \textit{Free completely discontinuous groups of affine transformations} (in Russian), Dokl. Akad. Nauk SSSR~272 (1983), p.~785--788.

\bibitem[Ma2]{mar84}
\textsc{G. A. Margulis}, \textit{Complete affine locally flat manifolds with a free fundamental group}, J. Soviet Math.~1934 (1987), p.~129--139, translated from Zap.~Naucha. Sem. Leningrad. Otdel. Mat. Inst. Steklov (LOMI)~134 (1984), p.~190--205.

\bibitem[Mi]{mil77}
\textsc{J. Milnor}, \textit{On fundamental groups of complete affinely flat manifolds}, Adv. Math.~25 (1977), p.~178--187.

\bibitem[R]{ran79}
\textsc{B. Randol}, \textit{Cylinders in Riemann surfaces}, Comment. Math. Helv.~54 (1979), p.~1--5.

\bibitem[T]{thu86}
\textsc{W. P. Thurston}, \textit{Minimal stretch maps between hyperbolic surfaces}, preprint (1986), arXiv:9801039.

\end{thebibliography}
\end{document}